\documentclass{amsart}
\usepackage{graphicx}
\usepackage{verbatim}
\usepackage{cite}
\usepackage{esint} 
\newtheorem{thm}{Theorem}[section]
\newtheorem{cor}[thm]{Corollary}
\newtheorem{lem}[thm]{Lemma}
\newtheorem{prop}[thm]{Proposition}
\theoremstyle{definition}
\newtheorem{defn}[thm]{Definition}
\theoremstyle{remark}
\newtheorem{rem}[thm]{Remark}
\newtheorem{claim}[thm]{Claim}
\numberwithin{equation}{section}
\newcommand{\mint}{\mathop{\int\hspace{-1.05em}{\--}}\nolimits}

\newcommand{\Real}{\mathbb R}

\newcommand{\func}[1]{\ensuremath{\mathrm{#1} \:} }

\newcommand{\dist}[0]{\mathrm{dist}}

\newcommand{\osc}[0]{\mathop{\mathrm{osc}}}

\newcommand{\spt}[0]{\func{spt}}
\newcommand{\Ss}[0]{{\mathbb S}}

\newcommand{\xX}[0]{\mathbf x}

\newcommand{\Bt}[0]{B_{2t}\backslash B_t}
\newcommand{\Bd}[0]{B_{\delta}\backslash B_{\frac \delta 2}}

\title[Compactness results for biharmonic maps]{Compactness results for sequences of approximate biharmonic maps}
\author[C.~Breiner]{Christine~Breiner}

\address[C.~Breiner]{Department of Mathematics, Fordham University, Bronx, NY 10458}
\email{cbreiner@fordham.edu}

\author[T.~Lamm]{Tobias~Lamm}
\address[T.~Lamm]{Institute for Analysis, Karlsruhe Institute of Technology (KIT), Kaiserstr. 89-93, D-76133 Karlsruhe, Germany}
\email{tobias.lamm@kit.edu}

\thanks{The first author was supported in part by NSF grant DMS-1308420 and an AMS-Simons Travel Grant.}

\keywords{Harmonic maps, biharmonic maps, bubbling, energy quantization}

\begin{document}

\maketitle

%\date{\today}

\begin{abstract}In this article, we prove energy quantization for approximate (intrinsic and extrinsic) biharmonic maps into spheres where the approximate map is in $L \log L$. Moreover, we demonstrate that if the $L\log L$ norm of the approximate maps does not concentrate, the image of the bubbles are connected without necks.
\end{abstract}
\section{Introduction}Critical points to the Dirichlet energy
\[
E(u) :=\frac 12 \int_{\Omega} |D u|^2dx
\]are called \emph{harmonic maps} and the compactness theory for such a sequence in two dimensions is well understood. Let $\Omega \subset \Real^2$ be a bounded domain and $N$ be a smooth, compact Riemannian manifold. For a sequence of harmonic maps $u_k \in W^{1,2}(\Omega, N)$ with uniform energy bounds, Sacks and Uhlenbeck \cite{SacksUhlenbeck} proved that a subsequence $u_k$ converges weakly to a harmonic $u_\infty$ on $\Omega$ and $u_k \to u_\infty$ in $C^\infty(\Omega \backslash \{x_1, \dots, x_\ell\})$ for some finite $\ell$ depending on the energy bound. For each $x_i$, Sacks and Uhlenbeck showed there exist some number of ``bubbles'', maps $\phi_{ij}: \Ss^2 \to N$, that result from appropriate conformal scalings of the sequence $u_k$ near $x_i$. In dimension two $E(u)$ is conformally invariant and thus one can ask whether any energy is lost in the limit. Jost \cite{Jost2DBook} proved that in fact the energy is quantized; there is no unaccounted energy loss:
\[
\lim_{k \to \infty} E(u_k) = E(u_\infty) + \sum_{i=1}^\ell \sum_{j=1}^{\ell_i} E(\phi_{ij}).
\]Parker \cite{Parker} provided the complete description of the $C^0$ limit or ``bubble tree''. In particular, he demonstrated that the images of the limiting map $u_\infty$ and the bubbles $\phi_{ij}$ are connected without necks. Around the same time various authors proved energy quantization and the no-neck property for approximate harmonic maps \cite{DingTian,QingTian,ChenTian,LinWang,WangPS}.

In this paper we are interested in an analogous compactness problem for a scale invariant energy in four dimensions. Let $(M^4, g)$ and $(N^k,h)$ be compact Riemannian manifolds without boundary, with $N^k$ isometrically embedded in some $\Real^n$. Consider the energy functional
\[
E_{ext}(u) := \int_{M}|\Delta u|^2dx
\]for $u \in W^{2,2}(M, N)$ where $\Delta$ is the Laplace-Beltrami operator. Critical points to this functional are called \emph{extrinsic biharmonic maps} and the Euler-Lagrange equation satisfied by such maps is of fourth order. Clearly, this functional depends upon the immersion of $N$ into $\Real^n$. To avoid such a dependence, one may instead consider critical points to the functional
\[
E_{int}(u):= \int_M |(\Delta u)^T|^2 dx
\]where $(\Delta u)^T$ is the projection of $\Delta u$ onto $T_uN$. Critical points to this functional are called \emph{intrinsic biharmonic maps}. The Euler-Lagrange equations satisfied by extrinsic and intrinsic biharmonic maps have been computed (see for instance \cite{Wang2}). We will be interested in approximate critical points.
\begin{defn}
Let $u \in W^{2,2}(B_1, N)$ where $B_1 \subset \Real^4$ and $N$ is a $C^3$ closed submanifold of some $\Real^n$. Let $f \in L\log L(B_1, \Real^n)$. Then $u$ is an \emph{$f$-approximate biharmonic map} if
\[
\Delta^2 u - \Delta(A(u)(D u, D u)) - 2 d^*\langle \Delta u, D P(u)\rangle + \langle \Delta(P(u)), \Delta u\rangle = f.
\]We define $u$ to be an \emph{$f$-approximate intrinsic biharmonic map} if
\begin{align*}
&\Delta^2 u - \Delta(A(u)(D u, D u)) - 2 d^* \langle \Delta u, D P(u)\rangle + \langle \Delta(P(u)), \Delta u\rangle  \\&-P(u)(A(u)(D u, D u)D_u A(u)(D u, D u)) - 2A(u)(D u, D u)A(u)(D u, D P(u))=f.
\end{align*}
Here $A$ is the second fundamental form of $N \hookrightarrow \Real^n$ and $P(u): \Real^{n} \to T_uN$ is the orthogonal projection from $\Real^n$ to the tangent space of $N$ at $u$.
\end{defn}
Recently, Hornung-Moser \cite{HornungMoser}, Laurain-Rivi\`ere \cite{LaurainRiviere}, and Wang-Zheng \cite{WangZheng} determined the energy quantization result for sequences of intrinsic biharmonic maps, approximate intrinsic and extrinsic biharmonic maps, and approximate extrinsic biharmonic maps respectively. (In fact, the result of \cite{LaurainRiviere} applies to a broader class of solutions to scaling invariant variational problems in dimension four.)

As a first result, we demonstrate that when the target manifold is a sphere, the energy quantization result extends to $f$-approximate biharmonic maps with $f \in L \log L$. For the definition of this Banach space, see the appendix.

\begin{thm}\label{bubblethm}
Let $f_k\in L\log L(B_1, \Real^{n+1})$ and $u_k \in W^{2,2}(B_1, \mathbb S^n)$ be a sequence of $f_k$-approximate biharmonic maps with 
\begin{equation}
\label{globalest}
\|D ^2u_k\|_{L^2(B_1)}+\|D u_k\|_{L^4(B_1)}+ \|f_k\|_{L \log L(B_1)} \leq \Lambda < \infty.
\end{equation}
If $u_k\rightharpoonup u$ weakly in $W^{2,2}(B_1,\mathbb S^n)$, there exists $ \{x_1, \dots, x_\ell\} \subset B_1$ such that $u_k \to u$ in $W^{2,2}_{loc} (B_1\backslash \{x_1, \dots, x_\ell\},\mathbb S^n)$.

Moreover, for each $1 \leq i \leq \ell$ there exists $\ell_i \in \mathbb N$ and nontrivial, smooth biharmonic maps $\omega_{ij} \in C^{\infty}(\mathbb R^4, \mathbb S^n)$ with finite energy ($1 \leq j \leq \ell_i$) such that
\[
\lim_{k \to \infty}\int_{B_{r_i}(x_i)} |D^2 u_k|^2 = \int_{B_{r_i}(x_i)}|D^2 u|^2 + \sum_{j=1}^{\ell_i} \int_{\mathbb R^4} |D^2 \omega_{ij}|^2,
\]
\[
\lim_{k \to \infty}\int_{B_{r_i}(x_i)} |D u_k|^4 = \int_{B_{r_i}(x_i)}|D u|^4 + \sum_{j=1}^{\ell_i} \int_{\mathbb R^4} |D \omega_{ij}|^4.
\]Here $r_i = \frac 12 \min_{1 \leq j \leq \ell, j \neq i}\{|x_i-x_j|, \dist(x_i, \partial B_1)\}$.\end{thm}
As a second result, we demonstrate the no neck property for approximate biharmonic maps with the approximating functions $L\log L$ norm not concentrating. 
\begin{thm}\label{neckthm}
Let $f_k \in L\log L$ such that the $L \log L$ norm does not concentrate. For $u_k$ a sequence of $f_k$-approximate biharmonic maps satisfying \eqref{globalest}, the images of $u$ and the maps $\omega_{ij}$ described above are connected in $\Ss^n$ without necks.
\end{thm}In particular, if $f_k \in \phi(L)$, an Orlicz space such that $\lim_{t \to \infty} \frac{\phi(t)}{t \log t} =  \infty$, the theorem holds. For a definition of an Orlicz space, see the appendix.
\begin{rem}
The theorems also hold for $u_k$ a sequence of $f_k$-approximate intrinsic biharmonic maps. We will prove the theorems in detail for $f_k$-approximate biharmonic maps and point out the necessary changes one must make to prove the intrinsic case.
\end{rem}
We consider biharmonic maps into spheres because the symmetry of the target provides structure to the equation that can be exploited to prove higher regularity.  For an $f$-approximate biharmonic map into $\Ss^n$, the structural equations takes the form (see Wang \cite{Wang})
\begin{equation}\label{biharmoniceq}
d^*(D\Delta u \wedge u - \Delta u \wedge Du) =f\wedge u
\end{equation} and for $f$-approximate intrinsic biharmonic $u$
\begin{equation}
d^*(D\Delta u \wedge u - \Delta u \wedge Du+ 2|Du|^2 Du \wedge u) =f\wedge u.
\end{equation}

The structure of the equation for harmonic maps from a compact Riemann surface into $\mathbb S^n$ was determined independently by Chen \cite{Chen} and Shatah \cite{Shatah}. They demonstrated that $u$ satisfies the conservation law
\[
d^*(Du \wedge u) =0.
\]H\'elein \cite{Helein} used the structure of this equation and Wente's inequality \cite{Wente} to determine that any weakly harmonic $u \in W^{1,2}$ was in fact $C^\infty$. 

In a recent paper by Li and Zhu \cite{LiZhu}, the authors use this additional structure to determine energy quantization for approximate harmonic maps.
In their setting the equation takes the form $d^*(Du \wedge u) = \tau \wedge u$ for $\tau\in L \log L$. Our proof of energy quantization is similar in spirit to the work of \cite{LiZhu} and to the recent small energy compactness result of Sharp and Topping \cite{ST}. Of critical importance are the energy estimates we prove in Section \ref{energy_section}. The first estimates, from Proposition \ref{thirdderivs} are used in two ways. First, the $L^p$ estimates of \eqref{epsreg2}, \eqref{epsreg3} provide sufficient control to determine a small energy compactness result away from the bubbles. Second, we use Lorentz space duality to prove energy quantization and thus require uniform bounds on the appropriate Lorentz energies as in \eqref{epsreg1}. In Section \ref{quantize_section} we prove the energy quantization result. We point out that as the oscillation bound contains an energy term of the form $\|D\Delta u_k\|_{L^{\frac 43}}$, we must also prove this energy is quantized. This point justifies the necessity of the estimate \eqref{epsreg4}. We prove the energy quantization result, under the presumption of the occurrence of one bubble, in Proposition \ref{quantizeprop}.

We next use this stronger energy quantization result for maps into spheres to prove a no-neck property. Zhu \cite{Zhu} showed the no-neck property for approximate harmonic maps with $\tau$ in a space essentially between $L^p$ with $p>1$ and $L \log L$. 
For $w$, a cutoff function of the approximate harmonic map $u$, Zhu considers a Hodge decomposition of the one-form $\beta:=Dw \wedge u$. (This is actually a matrix of one-forms but we gloss over that point for now.) He bounds $\|\beta\|_{L^{2,1}}$ by bounding each component of the decomposition and uses this to bound $\|D w\|_{L^{2,1}}$ by $\|D u\|_{L^2}$ plus a norm of the torsion term, $\tau$. Using $\varepsilon$-compactness and a simple duality argument, he shows the oscillation of $u$ is controlled by $\|D w\|_{L^{2,1}}$, which in turn implies the desired result. 

Like Zhu, we prove the no-neck property by demonstrating that the oscillation of an $f$-approximate biharmonic map is controlled by norms that tend to zero in the neck region. Using a duality argument, we first determine that the oscillation of $u$ on an annular region is bounded by quantized energy terms plus a third derivative of a cutoff function $w$. Our main work is in determining an appropriate estimate for $\|D\Delta w\|_{L^{{\frac 43},1}}.$ We determine this bound by considering the one form $\beta = D\Delta w \wedge u - \Delta w \wedge Du$, and we bound $D\Delta w$ by bounding $\beta$ via its Hodge decomposition. In particular, we take advantage of the divergence structure of the equation for biharmonic maps into spheres to show that $\beta$ not only has good $L^{{\frac 43}}$ estimates but in fact has good estimates in $L^{{\frac 43},1}$. This second estimate allows us to prove the necessary oscillation lemma. The proof of the oscillation lemma constitutes the work of Section \ref{oscillation_section}. Coupling the oscillation lemma with energy quantization, we prove Theorem \ref{neckthm} in Section \ref{no_neck_section}.

Finally, the arguments we use require a familiarity with Lorentz spaces and the appropriate embedding theorems relevant in dimension four. In the appendix, we describe the various Banach spaces and collect the necessary embeddings and estimates.

Many steps of the proof require the use of cutoff functions so we set the following notation.
\begin{defn}
Let $\phi \in C^\infty_0(B_2)$ with $\phi\equiv 1$ in $B_1$. For all $r>0$, define $\phi_r(\xX) = \phi(\frac \xX r)$. 
\end{defn}

\emph{Note added in proof:} As we finalized the paper, we noticed a somewhat related preprint posted on the arxiv by Liu and Yin (arXiv:1312.4600v1), in which they claim the no-neck property holds for sequences of biharmonic maps into general targets. Their methods are quite different from ours and we believe our results are of independent interest.

\section{Energy Estimates}\label{energy_section}
To establish strong convergence away from points of energy concentration, we first prove the necessary energy estimates. The small energy compactness result relies on the fact that in both \eqref{epsreg2}, \eqref{epsreg3} there is an extra power of the energy on the right hand side of the inequality. Thus, small energy implies that $\|Du_k\|_{L^4}$ and $\|D^2u_k\|_{L^2}$ must converge to zero on small balls. Measure theory arguments in the next section will then imply strong convergence for these norms to some $Du, D^2u$ respectively.

\begin{prop}\label{thirdderivs}Let $u \in W^{2,2}(B_{2}, \Ss^n)$ be an $f$-approximate (intrinsic) biharmonic map where $f \in L \log L (B_{2}, \Real^{n+1})$.
Then there exists $C >0$ such that
\begin{align}
\|D^3 u\|_{L^{\frac 43,1}(B_1)}&+\|D^2 u\|_{L^{2,1}(B_1)}+\|Du\|_{L^{4,1}(B_1)} \nonumber \\
&\leq C\left( \|D^2u\|^2_{L^{2}(B_{2})}+\|Du\|^2_{L^2(B_{2})}+\|Du\|_{L^2(B_2)}+ \|f\|_{L\log L(B_{2})}\right).\label{epsreg1}
\end{align}
Moreover, there exists $\tilde \varepsilon>0$ such that if
\[
\|D^2u\|_{L^2(B_{2})} + \|Du\|_{L^4(B_{2})} <\tilde \varepsilon
\]then for every $0<r<\frac 12$,
\begin{align}
\|D^2 u\|^2_{L^{2}(B_r)}&\leq Cr^2\| D^2u\|^2_{L^{2}(B_2)}\nonumber \\
&+C\left(\|D^2u\|^4_{L^{2}(B_{2})}+\|Du\|^4_{L^4(B_{2})}+ \|f\|_{L^1(B_2)}^{2}\|f\|_{L\log L(B_2)}\right), \label{epsreg2}
\end{align}
\begin{align}
\|Du\|^4_{L^{4}(B_r)}&\leq Cr^4\| Du\|^4_{L^{2}(B_2)}\nonumber \\
&+C\left(\|D^2u\|^8_{L^{2}(B_{2})}+\|Du\|^8_{L^4(B_{2})}+ \|f\|_{L^1(B_2)}^{3}\|f\|_{L\log L(B_2)}\right) \label{epsreg3}
\end{align}
and
\begin{align}
\|D\Delta u\|^{\frac43}_{L^{\frac43}(B_r)}&\leq Cr^{\frac 43}\| D^2u\|^{\frac43}_{L^{2}(B_2)}\nonumber \\
&+C\left(\|D^2u\|^{\frac83}_{L^{2}(B_{2})}+\|Du\|^{\frac83}_{L^4(B_{2})}+ \|f\|_{L^1(B_2)}^{\frac 13}\|f\|_{L\log L(B_2)}\right). \label{epsreg4}
\end{align}
\end{prop}
\begin{rem}
In point of fact, we do not need the full strength of \eqref{epsreg4} in application. We use instead the estimate
\[
\|D\Delta u\|_{L^{\frac 43}(B_r)}^{\frac 43}\leq C\left(\|D^2 u\|_{L^2(B_{8r})}^{\frac 43} + \|Du\|_{L^4(B_{8r})}^{\frac 43}+ \|f\|_{L^1(B_{8r})}^{\frac 13}\|f\|_{L \log L(B_{8r})}\right),
\]which can be immediately proven via the method outlined below.
\end{rem}
\begin{proof}
First, find $v \in W^{1,2}_0(B_2, so(n+1)) \cap W^{2,2}(B_2,so(n+1))$ such that
\[
\Delta v = \Delta u \wedge u.
\]Thus, for each $i,j \in \{1, \dots, n+1\}$, $\Delta v^{ij} = u^j\Delta u^i - u^i\Delta u^j$. 
It follows from equation \eqref{biharmoniceq} that
\[
\Delta^2 v =\Delta (\Delta u \wedge u) = 2d^* (\Delta u \wedge Du)+f\wedge u.
\]
Next we let $\phi \in W^{2,2}_0(B_2,so({n+1})\otimes \Omega^1\Real^4)$ be the solution of
\[
\Delta^2 \phi = d^* \left( 2\Delta u \wedge D u \right).
\] Here $so({n+1})\otimes \Omega^1\Real^4$ denotes the space of $1$-forms tensored with $(n+1)\times (n+1)$-anti-symmetric matrices. 
Using Calder\'on-Zygmund theory coupled with interpolation, and using the estimates from Appendix \ref{estap} we determine that
\begin{equation}\label{phiest1}
\|D^3 \phi\|_{L^{\frac 43,1}(B_2)}+\|D^2 \phi\|_{L^{2,1}(B_2)}+\|D \phi\|_{L^{4,1}(B_2)} \leq c(\|D^2u\|^2_{L^{2}(B_{2})}+\|Du\|^2_{L^2(B_{2})}).
\end{equation}
Moreover, letting $\psi \in W^{2,2}_0(B_2,so({n+1}))$ be the solution of
\[
\Delta^2 \psi = f\wedge u,
\]
we conclude that
\begin{equation}\label{psiest1}
\|D \psi\|_{L^{4,1}(B_2)}+\|D^2 \psi \|_{L^{2,1}(B_2)}+\| D^3 \psi\|_{L^{\frac43,1}(B_2)}  \le c\|f\|_{L\log L(B_{2})}.
\end{equation}

Defining
\begin{align*}
B:=  v -\phi-\psi \label{bih}
\end{align*}
and using the above equation for $v$, we conclude that each $B^{ij}$ is a biharmonic function on $B_2$. Now every biharmonic function satisfies the mean value property
\[
B(x)= c_1 \mint_{B_r(x)} B(y)dy- c_2 \mint_{B_{2r}(x)} B(y) dy,
\]
for every $B_{2r}(x) \subset B_2$ (see e.g. \cite{huilgol}). Hence we estimate
\begin{align*}
\|D^2 B\|_{L^{2,1}(B_{3/2})}&+\|D^3B\|_{L^{\frac 43,1}(B_{3/2})}\\ &\le c\|DB\|_{L^{2}(B_2)}\\
&\le c (\|Dv\|_{L^{2}(B_2)}+\|f\|_{L\log L(B_{2})}+\|D^2 u\|^2_{L^2(B_2)}+ \|Du\|^2_{L^{2}(B_2)}).
\end{align*}Since $v=0$ on $\partial B_2$ we can use the divergence theorem and Cauchy-Schwarz to show that
\[
\int_{B_2} |Dv^{ij}|^2 = -\int_{B_2} v^{ij}\Delta v^{ij} = -\int_{B_2}Dv^{ij} \cdot (Du \wedge u)^{ij} \leq \frac 12 \int_{B_2} |Dv^{ij}|^2 + C \int_{B_2} |Du|^2.
\]
Thus,
\begin{align*}
\|D^2 B\|_{L^{2,1}(B_{3/2})}+&\|D^3B\|_{L^{\frac 43,1}(B_{3/2})}\\&
\quad  \le c\left(\|Du\|_{L^{2}(B_2)}+\|f\|_{L\log L(B_{2})}+\|D^2 u\|^2_{L^2(B_2)}+ \|Du\|^2_{L^{2}(B_2)}\right).
\end{align*}

Now we observe that as $\Delta v = \Delta u \wedge u$,
\[
\Delta u = (\Delta u \wedge u).u + \langle \Delta u, u\rangle u = \Delta v.u - |Du|^2 u
\]where here $\Omega.u$ represents matrix multiplication.
Therefore,
\[
\Delta^2 u = \Delta(\Delta v.u- |Du|^2u) = d^*\left(D\Delta v.u + \Delta v.Du -D(|Du|^2u)\right).
\]To get the second and third derivative estimate in \eqref{epsreg1}, we first observe that
\[
\|D^2v\|_{L^{2,1}(B_{3/2})}+\|D^3 v\|_{L^{\frac 43,1}(B_{3/2})} \leq  c \left(\|Du\|_{L^{2}(B_2)}+\|f\|_{L\log L(B_{2})}+\|D^2 u\|^2_{L^2(B_2)}+ \|Du\|^2_{L^{2}(B_2)}\right).
\]Using the previous estimates and Appendix \ref{estap}, we observe that the one form in the parentheses is in $L^{\frac 43,1}$. Lemma A.3 from \cite{LammRiv} implies that
\begin{align*}
\|D^2u\|_{L^{2,1}(B_1)}+\|D^3 u\|_{L^{\frac 43,1}(B_1)}& \leq c\left( \|D^3 v\|_{L^{\frac 43,1}(B_{3/2})} + \|D^2 v\|^2_{L^{2}(B_2)} + \|D^2u\|^2_{L^2(B_2)} + \|Du\|^2_{L^2(B_2)}\right)\\
& \leq c\left(\|Du\|_{L^{2}(B_2)}+\|f\|_{L\log L(B_{2})}+\|D^2 u\|^2_{L^2(B_2)}+ \|Du\|^2_{L^{2}(B_2)} \right).
\end{align*}
Finally, Sobolev embedding for Lorentz spaces implies that
\[
\|Du\|_{L^{4,1}(B_1)} \leq c\left(\|D^2u\|_{L^{2,1}(B_2)}+\|Du\|_{L^{2,1}(B_2)}\right) \leq c\left(\|D^2u\|_{L^{2,1}(B_2)}+\|Du\|_{L^2(B_2)}\right).
\]Combining this with the previous estimates finishes the proof of \eqref{epsreg1}.

To prove the small energy estimates, we observe that $u$ satisfies (see for instance \cite{LammRiv}, equations 1.4, 1.14)
\begin{equation}\label{ELu}
\Delta^2 u = \Delta(V\cdot D u) + d^*(wD u) + W \cdot D u + f
\end{equation}where 
$V^{ij}=u^iD u^j - u^jD u^i,$ $w^{ij} = -d^*(V^{ij}) -2|D u|^2\delta_{ij},$ %, \qquad \text{and }  D\tilde B =W
and $W^{ij} = -D(d^*(V^{ij})) +2 (\Delta u^i D u^j -\Delta u^j D u^i)$. Let $\mathcal M_m$ denote the space of $m\times m$ matrices and $\mathcal M_m \otimes \Omega^k \Real^4$ denote the space of $k$-forms tensored with $m \times m$ matrices. Then $V \in W^{1,2}(B_{2}, \mathcal M_{n+1} \otimes \Omega^1 \Real^4)$, $w \in L^2(B_{2}, \mathcal M_{n+1})$, $W \in W^{-1,2}(B_{2}, \mathcal M_{n+1} \otimes \Omega^1 \Real^4)$.  
%In point of fact, the equations mentioned in \cite{LammRiv} are related to vector fields. We prefer to work with forms and take advantage of the fact that the operator $curl$ on two-vector fields is the same as the operator $d^*$ on the corresponding two-form and $curl$ on one-vector fields is the same as $d$ on the corresponding one-form. As we will need a bound on $E$, we first consider that here. 

Without loss of generality we extend $f$ by zero outside of $B_2$.
The small energy hypothesis implies (see for instance \cite{LammRiv}) that there exist $A\in L^\infty\cap W^{2,2}(B_{1},GL_{n+1}),\tilde B \in W^{1,\frac 43}(B_1,\mathcal M_{n+1} \otimes \Omega^2\Real^4)$ such that
\[
D\Delta A + \Delta AV - DA w + AW= D\tilde B
\]and
\begin{align*}
\Delta(A\Delta u)&=d^*\left( 2DA\Delta u - \Delta ADu+AwD-DA(V\cdot Du)+AD(V\cdot Du)+\tilde B\cdot Du)\right)+Af\\
&:= d^*(K) + Af.
\end{align*}
Moreover,
\[
\|DA\|_{W^{1,2}(B_{1})} + \|dist(A,SO(n+1))\|_{L^\infty(B_{1})} + \|\tilde B\|_{W^{1,\frac 43}(B_{1})} \leq c\left(\|D^2u\|_{L^2(B_{2})}+ \|Du\|_{L^4(B_{2})}\right).
\]

First, we determine $E,F \in W^{1,2}_0(B_{1})$ such that
\[
\Delta E= d^*(K), \quad \quad \Delta F = Af.
\]
Interpolating on standard $L^p$ theory, we get the estimates
\begin{align*}
\|E\|_{L^{2,1}(B_1)}+\|DE\|_{L^{\frac 43,1}(B_1)}& \leq c\|K\|_{L^{\frac 43,1}(B_{2})}\\
&\leq c\left(\|D^2u\|_{L^2(B_{2})}^2+\|Du\|_{L^4(B_{2})}^2  \right).
\end{align*}Note that the estimate on $K$ comes from considering the form of the equation \eqref{ELu} and the estimates on $V,w,W$ and consequently those on $A,\tilde B$.

To determine estimates on $F$, we first observe that the estimates of Appendix \ref{estap} imply that for $G$ the fundamental solution to $\Delta^2 G = \delta_0$,
\[
\|F\|_{L^{2,\infty}(B_{1})}\leq c\|D^2G\ast (Af)\|_{L^{2,\infty}(B_{1})} \leq c \|f\|_{L^1(B_{2})},
\]
\[
\|DF\|_{L^{\frac 43,\infty}(B_1)} \leq c\|D^3G\|_{L^{\frac 43,\infty}(B_{2})}\|f\|_{L^1(B_{2})}.
\]
Also, since $\Delta F = Af \in \mathcal H^1(\Real^4)$, standard theory implies that $D^2F \in L^1(\Real^4)$ and thus by the embedding of $W^{1,1}$ into $L^{\frac 43,1}$ and Sobolev embeddings in $\Real^4$, 
\[
 \|F\|_{L^{2,1}(B_1)}+\|DF\|_{L^{\frac 43,1}(B_1)} \leq c\|f\|_{L\log L(B_{2})}.
\]Using a duality argument, we conclude that
\begin{align*}
\|F\|_{L^2(B_1)}^2& \leq c\|F\|_{L^{2,\infty}(B_1)}\|F\|_{L^{2,1}(B_1)}\\
&\leq c\|f\|_{L^1(B_{2})}\|f\|_{L\log L(B_{2})}
\end{align*}and
\begin{align*}
\|DF\|_{L^{\frac 43}(B_1)}^{\frac 43}& \leq c\|(DF)^{\frac 13}\|_{L^{4,\infty(B_1)}}\|DF\|_{L^{\frac 43,1}(B_1)}\\
&\leq c\|DF\|_{L^{\frac 43,\infty}(B_2)}^{\frac 13}\|f\|_{L \log L(B_2)}\\
&\leq c\|f\|_{L^1(B_{2})}^{\frac 13}\|f\|_{L\log L(B_{2})}.
\end{align*}

Now, set $H = A\Delta u -E-F$. Then $\Delta H=0$ in $B_1$ and using standard estimates on harmonic functions we determine that for all $0<r<\frac 12$,
\[
\|H\|_{L^2(B_r)} +\|DH\|_{L^{\frac 43}(B_r)} \leq cr\|H\|_{W^{1,\infty}(B_{1/2})}\leq cr\|H\|_{L^2(B_1)}.
\]
The previous estimates imply that
\[
\|H\|_{L^2(B_1)}^2\leq c\left(\|D^2u\|_{L^2(B_2)}^2+\|Du\|_{L^4(B_2)}^4+ \|f\|_{L^1(B_2)}\|f\|_{L\log L(B_2)} \right).
\]
Since,
\[
\Delta u = A^{-1}(E+F+H),
\]the estimates for $D^2u$ now follow from a standard cutoff argument and the previous estimates.

We estimate $\|D\Delta u\|_{L^{\frac 43}(B_r)}$ by using the previous estimates and noting that
\[
\|D(A^{-1}(E+F+H))\|_{L^{\frac 43}(B_r)} \leq C\left(\|E+F+H\|_{L^2(B_r)}\|DA\|_{L^4(B_r)}+ \|D(E+F+H)\|_{L^{\frac 43}(B_r)}\right).
\]

To determine an estimate for $Du$, we first consider $\alpha \in W^{2,2}(B_1)$, $\beta \in W^{1,2}_0(B_1,\Omega^1\Real^4)\cap W^{2,2}$ such that
\[
Adu = d\alpha + d^* \beta.
\]Then
\[
 \Delta^2 \alpha = \Delta d^*(Adu) =\Delta(A\Delta u + \nabla A \cdot \nabla u) = d^*(\tilde K) + Af \text{ on }B_1
\]and
\[
\Delta \beta = dA \wedge du \text{ on }B_1.
\]Here $\tilde K$ is the appropriate modification of $K$ to include the additional term. We first observe that
\[
\|D\beta\|_{L^4(B_{r})} \leq  c\left(\|D^2\beta\|_{L^2(B_1)}+ \|D\beta\|_{L^2(B_1)}\right).
\]Standard $L^p$ theory implies that
\[
\|D^2\beta\|_{L^2(B_{2})} \leq c\|DA\|_{W^{1,2}(B_1)}\|Du\|_{W^{1,2}(B_1)}.
\]Moreover, using a weighted Cauchy-Schwarz and the Poincar\'e Inequality, we note that
\[
\int_{B_1}|\nabla \beta^{ij}|^2 = - \int_{B_1} \beta^{ij}(dA \wedge du)^{ij}\leq c\|DA\|_{L^4(B_1)}^2\|Du\|_{L^4(B_1)}^2+ \frac 12 \|\nabla \beta\|_{L^2(B_1)}^2.
\]Combining this with previous estimates implies that
\[
\|D\beta\|_{L^4(B_{r})} \leq  c\left( \|D^2u\|_{L^2(B_{2})}^2+ \|Du\|_{L^4(B_{2})}^2\right)
\]For the $\alpha$ term, we follow the ideas used to prove \eqref{epsreg1}.  Indeed, first determine $\phi, \psi \in W^{2,2}_0(B_2)$ such that $\Delta^2 \phi = d^*(K)$ and $\Delta^2 \psi = Af$. Then by \eqref{phiest1}, \eqref{psiest1}, and appropriate duality arguments, we conclude that for any $0<r<1$,
\[
\|D\phi\|_{L^4(B_{r})}\leq c\left(\|D^2u\|_{L^2(B_2)}^2 + \|Du\|_{L^4(B_2)}^2 \right),
\]
\[
  \|D\psi\|_{L^4(B_{r})}^4 \leq c\|f\|_{L^1(B_2)}^3\|f\|_{L\log L(B_2)}.
\]Setting $B=\alpha - \psi - \phi$, $\Delta^2 B=0$ on $B_1$ and we use the mean value property to show that for any $0<r<\frac 12$,
\[
\|DB\|_{L^4(B_r)} \leq cr\|DB\|_{L^\infty(B_{3/4})}\leq cr\|DB\|_{L^4(B_{7/8})}.
\]Noting that
\begin{align*}
\|DB\|_{L^4(B_{7/8})}^4\leq c\left( \|D\alpha\|_{L^4(B_{7/8})}^4\right.& + \|Du\|_{L^4(B_{1})}^4+\|D^2u\|_{L^2(B_2)}^8\\
&\left. + \|Du\|_{L^4(B_2)}^8+\|f\|_{L^1(B_2)}^{3}\|f\|_{L\log L(B_2)} \right)
\end{align*}we combine the previous estimates to get the result for $Du$.
\end{proof}

\begin{rem}
When $u$ is intrinsic, the strategy is the same except for two things. In the first part of the argument, the equation for $u$ has the additional term $-d^*(|D u|^2 Du\wedge u)$ on the right hand side. But this term doesn't change the estimates. In the second part of the argument, $W^{ij}$ includes the term $|D u|^2(u^iD u^j-u^j D u^i)$. This gives the same value for $d^*(W^{ij})$ and all estimates going forward are the same.
\end{rem}
We will prove the energy quantization results by appealing to Lorentz duality. In Proposition \ref{thirdderivs}, we determined uniform estimates for Lorentz norms of the form $L^{p,1}$. The next lemma provides the necessary small energy estimates for the $L^{p,\infty}$ norms on the annular region, presuming small energy on all dyadic annuli.

\begin{lem}\label{inftylemma}
Let $u \in W^{2,2}(B_1,\Ss^n)$ be an $f$-approximate biharmonic map with $f \in L \log L(B_1, \Real^{n+1})$. Given $\varepsilon>0$, suppose that for all $\rho$ such that $B_{2\rho} \backslash B_{\rho} \subset B_{2\delta} \backslash B_{t/2}$
\begin{equation}\int_{B_{2\rho}\backslash B_\rho} |Du|^4 + |D^2 u|^2 + |D\Delta u|^{{\frac 43}} < \varepsilon.
\end{equation}Then,
\begin{equation*}
\|Du\|_{L^{4,\infty}(B_\delta \backslash B_t)}+
\label{inftyests}\|D^2 u\|_{L^{2,\infty}(B_\delta \backslash B_t)}+
\|D\Delta u\|_{L^{{\frac 43},\infty}(B_\delta \backslash B_t)} \leq C\left(\varepsilon^{\frac 18} + \left( \log(1/\delta)\right)^{-1}\right).
\end{equation*}
\end{lem}
\begin{proof}
Let $\tilde \phi_k:= \phi_{2^{k+2}t} (1-\phi_{2^{k-2}t})$ be the annular cutoff supported on $A_k:=B_{2^{k+3}t} \backslash B_{2^{k-2}t}$ which is identically $1$ on $B_{2^{k+2}t}\backslash B_{2^{k-1}t}$. 
Let $G$ be the distribution such that $\Delta^2 G = \delta_0$ in $\Real^4$. Then $|DG(x) |= C|x|^{-1}$. Note that operator bounds on $D^kG$ can be found in the appendix.
Let $\overline u_k:= \mint_{A_k} u$.
Define $\tilde u_k(x):= \tilde \phi_k(u-\overline u_k)(x)$. Therefore on $B_{2^{k+1}t} \backslash B_{2^k t}$,
 \begin{align*}
 \Delta^2 \tilde u_k =(\Delta^2  \tilde \phi_k)(u-\overline u_k)+4 D\Delta \tilde \phi_k\cdot D(u-\overline u_k)+ 2\Delta \tilde \phi_k\Delta u + 4D  \tilde \phi_k\cdot D\Delta u+  \tilde \phi_k\Delta^2 u.
 \end{align*}Using the fact that $\Delta^2 u = \Delta(\Delta u \wedge u.u -|Du|^2u)$ and that $\Delta^2 u \wedge u = f\wedge u$, we note that
 \begin{align*}
 \tilde \phi_k\Delta^2 u =& d^*\left( \tilde \phi_k(2\Delta u \wedge D u.u+ 2\Delta u \wedge u.Du- D(u|Du|^2) )\right)\\
&\quad -D\tilde \phi_k \cdot(2\Delta u \wedge D u.u+2\Delta u\wedge u.Du- D(u|Du|^2))\\
& \quad  +\tilde \phi_k\left(f \wedge u.u -2 \Delta u\wedge Du.Du- \Delta u \wedge u.\Delta u\right) .
 \end{align*}And thus,
 \begin{align*}
\Delta^2 \tilde u_k = & (\Delta^2  \tilde \phi_k)(u-\overline u_k)+ 4D\Delta \tilde \phi_k\cdot D(u-\overline u_k)+ 2\Delta\tilde \phi_k\Delta u + 4D  \tilde \phi_k\cdot D\Delta u\\
&\quad -D\tilde \phi_k \cdot(2\Delta u \wedge D u.u+2\Delta u\wedge u.Du- D(u|Du|^2))\\
& \quad +d^*\left( \tilde \phi_k(2\Delta u \wedge D u.u+ 2\Delta u \wedge u.Du- D(u|Du|^2) )\right)\\
& \quad  +\tilde \phi_k\left(f \wedge u.u -2 \Delta u\wedge Du.Du- \Delta u \wedge u.\Delta u\right).
\end{align*}
For ease of notation, we let $I_k$ denote the first four terms above, and $II_k, III_k, IV_k$ denote each of the last three terms. Then on each $B_{2^{k+1}t} \backslash B_{2^k t}$
\begin{align*}
|Du(x)|&=| D (\tilde \phi_k(u-\overline u_k))(x)|=|\Delta^2 G \ast D(\tilde \phi_k(u-\overline u_k))(x)|\\
&= |DG \ast \Delta^2(\tilde \phi_k(u-\overline u_k))(x)|=|DG \ast(I_k + II_k + III_k+IV_k)(x)|.
\end{align*} We consider each of these estimates separately. First, note that
\begin{align*}
|DG \ast I_k(x)| & \leq C\left|\int_{(B_{2^{k+3}t} \backslash B_{2^{k+2}t} )\cup (B_{2^{k-1}t} \backslash B_{2^{k-2}t})}\frac{1}{|x-y|}\left((2^kt)^{-4}(u-\overline u_k) \right. \right.\\
& \qquad \qquad \qquad \qquad\left.\left.+  (2^kt)^{-3}D(u-\overline u_k)+(2^kt)^{-2}\Delta u + (2^kt)^{-1}D\Delta u\right)dy\right|\\
& \leq C \left|\int_{(B_{2^{k+3}t} \backslash B_{2^{k-2}t})} (2^kt)^{-1} \left((2^kt)^{-4}(u-\overline u_k) +  (2^kt)^{-3}Du\right.\right.\\
& \qquad \qquad \qquad \qquad \qquad \qquad \qquad \qquad\qquad\left.\left.+(2^kt)^{-2}\Delta u + (2^kt)^{-1}D\Delta u\right) dy\right|\\
& \leq C \int_{(B_{2^{k+3}t} \backslash B_{2^{k-2}t})}(2^kt)^{-4}|Du| +(2^kt)^{-3}|D^2 u| + (2^kt)^{-2}|D\Delta u|\\
& \leq C(2^kt)^{-1} \left(\|Du\|_{L^4}+\|D^2u\|_{L^2} + \|D\Delta u\|_{L^{\frac 43}}\right)\\
& \leq C ({\varepsilon^{1/4}+\varepsilon^{1/2}+ \varepsilon^{3/4}}){|x|}^{-1}.
\end{align*}%where for the third to last inequality we used Poincar\'e and for the second to last inequality we used H\"older. 
Using the same ideas as previously, we bound
\begin{align*}
|DG \ast II_k(x)| &\leq C(2^kt)^{-2}\int_{(B_{2^{k+3}t} \backslash B_{2^{k-2}t})}\left|2\Delta u \wedge D u.u +2 \Delta u \wedge u.Du- D(u|Du|^2)\right|\\
& \leq C (2^kt)^{-1} \|2\Delta u \wedge D u.u+2\Delta u \wedge u.Du- D(u|Du|^2)\|_{L^{\frac 43}(B_{2^{k+3}t} \backslash B_{2^{k-2}t})}\\
& \leq C (2^kt)^{-1}\left(\|D^2u\|_{L^2} \|Du\|_{L^4} + \|Du\|_{L^4}^3 \right)\\
& \leq C({\varepsilon^{1/8} + \varepsilon^{3/4}}){|x|}^{-1}.
\end{align*}
Using the estimates from the appendix, we note that
\begin{align*}
\|DG \ast& III_k\|_{L^{4,\infty}(B_{2^{k+3}t} \backslash B_{2^{k-2}t})}\\ &\quad \quad\leq C \|D^2G\ast\tilde \phi_k\left(2\Delta u \wedge D u.u+2\Delta u \wedge u.Du- D(u|Du|^2)\right) \|_{L^{4, \infty}(B_{2^{k+3}t} \backslash B_{2^{k-2}t})}\\
&\quad \quad\leq C \|\tilde \phi_k\left(2\Delta u \wedge D u.u+2\Delta u \wedge u.Du- D(u|Du|^2)\right) \|_{L^{\frac 43}(B_{2^{k+3}t} \backslash B_{2^{k-2}t})}
\end{align*}and
\[
\|DG \ast IV_k\|_{L^{4,\infty}(B_{2^{k+3}t} \backslash B_{2^{k-2}t})} \leq C \|\tilde \phi_k(f \wedge u.u-2\Delta u \wedge Du.Du - \Delta u \wedge u.\Delta u)\|_{L^1(B_{2^{k+3}t} \backslash B_{2^{k-2}t})}.
\]
Thus
\begin{align*}
\left|\{x: \left|DG\ast(III_k\right.\right.&\left.\left.+IV_k)(x)\right|> \lambda\}\right|\\ &\leq \lambda^{-4}\|DG\ast(III_k+IV_k)\|_{L^{4,\infty}
(\Real^4)}^4 \\
& \leq C \lambda^{-4}\left(\| \tilde \phi_k\left(f \wedge u.u-2\Delta u \wedge Du.Du - \Delta u \wedge u.\Delta u\right)\|^4_{L^1(B_{2^{k+3}t} \backslash B_{2^{k-2}t})} \right.\\
& \quad + \left.\| 
\tilde \phi_k(2\Delta u \wedge D u.u + 2\Delta u \wedge u.Du- D(u|Du|^2) )\|_{L^{\frac 43}(B_{2^{k+3}t} \backslash B_{2^{k-2}t})}^4\right)\\
& \leq C \lambda^{-4} \left(\left(\int\tilde \phi_k|D^2u|^2 \right)^{2}\left(\int \tilde \phi_k |Du|^4\right) +\left( \int \tilde \phi_k|Du|^4\right)^3 \right)\\
& \qquad + C \lambda^{-4}\| \tilde \phi_k\left(f \wedge u.u-2\Delta u \wedge Du.Du - \Delta u \wedge u.\Delta u\right)\|^4_{L^1((B_{2^{k+3}t} \backslash B_{2^{k-2}t}))}.
\end{align*}

Thus, if $\delta = 2^M t$ then
\begin{align*}
|\{x \in B_\delta \backslash B_t: &|D u(x)| > 3\lambda\}|\\
& \leq  \sum_{k=0}^{M-1} |\{x \in B_{2^{k+1}t} \backslash B_{2^kt}: |D u(x)| >3 \lambda\}|\\
&\leq  \sum_{k=0}^{M-1}|\{x \in B_{2^{k+1}t} \backslash B_{2^kt}:|DG \ast I_k|>\lambda\}| \\
& \quad\quad+ \sum_{k=0}^{M-1}|\{x \in B_{2^{k+1}t} \backslash B_{2^kt}:|DG \ast II_k|>\lambda\}| \\
& \quad \quad+ \sum_{k=0}^{M-1}|\{x \in B_{2^{k+1}t} \backslash B_{2^kt}:|DG\ast(III_k + IV_k) |>\lambda\}|\\
& \leq \sum_{k=0}^{M-1}\left|\{ \left|DG\ast(III_k+IV_k)\right|> \lambda\}\right|+ |\{x \in B_1: C\frac{\varepsilon^{1/8}}{|x|} > \lambda\}| \\
& \leq C \lambda^{-4} \left( \varepsilon^{\frac 12}+\sum_{k=0}^{M-1}\| \tilde \phi_k\left(f \wedge u.u\right)\|^4_{L^1((B_{2^{k+3}t} \backslash B_{2^{k-2}t}))} \right.\\
& \quad \quad \left.+\sum_{k=0}^{M-1} \left(\left(\int\tilde \phi_k|D^2u|^2 \right)^{4}+\left(\int \tilde \phi_k |Du|^4\right)^{4}  \right.\right.\\
& \quad \quad\quad \left.+ \left.\left(\int \tilde \phi_k |Du|^4\right)^{3} +\left( \int \tilde \phi_k|Du|^4\right)^2 \right)\right)\\
& \leq C\lambda^{-4} \left( \varepsilon^{\frac 12}  + \left( \log(1/\delta)\right)^{-4}\|f\|^4_{L \log L(B_{2\delta})}+ \varepsilon^2\right).
\end{align*}For the estimate on $\|f \wedge u.u\|_{L^1}$ we use Lemma \ref{fL1lemma} and for the rest of the $L^1$ estimate we just use Cauchy-Schwarz. This proves the estimate for $Du$. The estimates for $D^2u, D\Delta u$ work in much the same way. In the case of $D^2u$, the terms like $III_k, VI_k$ require the fact that $D^3G:L^{\frac 43}\to L^{2,\infty}, D^2G:L^{1} \to L^{2,\infty}$ are bounded operators where the operation is convolution. For the term $D\Delta u$ we observe that $D^3G: L^1\to L^{\frac 43,\infty} , D^4G:L^{\frac 43} \to L^{\frac 43, \infty}$ are also bounded operators.
\end{proof}

\section{Energy Quantization -- Proof of Theorem \ref{bubblethm}}\label{quantize_section}
We now determine a weak convergence result which will give small energy compactness and help us complete the proof of the energy quantization. We follow the ideas of \cite{LiZhu,ST}, which in turn follow the arguments of \cite{Evans}, with appropriate minor modifications.  Throughout this lemma and its proof, we consider a measurable function $f$ as both a function and a Radon measure.
\begin{lem}\label{measures}
Suppose $\{V_k\} \subset W^{1,{\frac 43}}(B_1)$ is a bounded sequence in $B_1 \subset \Real^4$. Then there exist at most countable $\{x_i\} \subset B_1$ and $\{a_i >0\}$ with $\sum_i a_i < \infty$ and $V \in W^{1,{\frac 43}}(B_1)$ such that, after passing to a subsequence
\[
V_k^2 \rightharpoonup V^2+ \sum_i a_i \delta_{x_i}
\]weakly as measures.
\end{lem}
\begin{proof}
As $W^{1,{\frac 43}}$ embeds continuously into $L^2$ in four dimensions, after taking a subsequence, by Rellich compactness there exists some $V\in L^2$ such that $V_k \to V$ strongly in $L^p$ for $1\leq p<2$ and $V_k \rightharpoonup V$ weakly in $L^2$. Moreover, since $\{DV_k\}$ is uniformly bounded in $L^{{\frac 43}}$, it follows that $DV_k \rightharpoonup f \in L^{{\frac 43}}$ and $f$ is necessarily $DV$.

Set $g_k:= V_k - V$. Then $g_k \in L^2$ and $Dg_k \in L^{{\frac 43}}$ with uniform bounds. Thus, in the weak-$\ast$ topology, both $|Dg_k|^{{\frac 43}}$ and $g_k^2$ converge to non-negative Radon measures with finite total mass. (We denote this space $M(B)$). Then $g_k^2 \rightharpoonup \nu \in M(B)$ and $|Dg_k|^{{\frac 43}} \rightharpoonup \mu \in M(B)$ where $\nu,  \mu$ are both non-negative. Now consider $\phi \in C_0^1(B_1)$ and observe that the Sobolev embedding of $W^{1,{\frac 43}}$ into $L^2$ implies that
\[
\left(\int (\phi g_k)^2 dx \right)^{{\frac 12}} \leq C \left(\int |D(\phi g_k)|^{{\frac 43}}dx\right)^{{\frac 34}}.
\]Taking $k \to \infty$ and noting that $g_k \to 0$ in $L^{{\frac 43}}$ we use the weak convergence to observe that
\[
\int \phi^2 d\nu \leq C\left(\int |\phi|^{{\frac 43}} d\mu\right)^{{\frac 32}}.
\]Let $\phi$ approximate $\chi_{B_r(x)}$ for $B_r(x) \subset B_1$. Then 
\[\nu (B_r(x)) \leq C\left(\mu(B_r(x))\right)^{{\frac 32}}.\]
By standard results on the differentiation of measures (see \cite{EvansGariepy}, section 1.6), for any Borel set $E$
\[
\nu(E) =\int_E D_\mu \nu \, d\mu, \text{ where } D_\mu \nu(x) = \lim_{\substack{r \to 0}}\frac{\nu(B_r(x))}{\mu(B_r(x))} \text{ for } \mu-\text{ a.e. } x \in \Real^4.
\]Now, as $\mu$ is a finite, non-negative, Radon measure, there exist at most countably many $x_i \in B_1$ such that $\mu(\{x_i\})>0$. Moreover, for all $x\in B$ such that $\mu(\{x\})=0$, we note that
\[
D_\mu \nu(x) = \lim_{\substack{r \to 0}}\frac{\nu(B_r(x))}{\mu(B_r(x))} \leq C\lim_{\substack{r \to 0}}\mu(B_r(x))^{\frac 12}=0.
\]For every $x_j$ such that $\mu(\{x_j\})>0$, set $a_j = D_\mu \nu(x_j) \mu(\{x_j\})$. Then 
\[
\nu(E) = \int_E D_\mu \nu d\mu = \sum_{\substack{\{j | \,x_j \in E\}}}a_j \quad \text{ or}\quad \nu = \sum_j a_j \delta_{x_j}.
\]Since $g_k^2 \rightharpoonup \nu$ as measures, for $\phi \in C_0^0(B_1)$,
\[
\sum_j a_j \phi(x_j) = \lim_{k\to \infty}\int_{B_1} g_k^2 \phi dx = \lim_{k \to \infty}\int_{B_1}\left(V_k - V\right)^2 dx.
\]Since $(V_k-V)^2 = V_k^2 - V^2 + 2V(V-V_k)$ and $V-V_k = g_k \rightharpoonup 0$ in $L^2$, we have the result.
\end{proof}

\begin{cor}\label{measurecor}
For $\{V_k\}$ as in Lemma \ref{measures}, if
\begin{equation}\label{measurelimit}
\lim_{\substack{r \to 0}} \limsup_{\substack{k \to \infty}} \|V_k\|_{L^2(B_r(x))} =0
\end{equation}
for all $x \in B$, then 
\[
V_k \to V \text{ strongly in } L^2_{loc}(B).
\]
\end{cor}
\begin{proof}
Notice the condition \eqref{measurelimit} implies that $|V_k|^2 \rightharpoonup |V|^2$ weakly as bounded, Radon measures. Then (by \cite{EvansGariepy}, Section 1.9) for any $B_r(x) \subset B_1$,  $\|V_k\|_{L^2(B_r(x))}\to \|V\|_{L^2(B_r(x))}$ strongly for all $B_r(x) \subset B_1$. Then, again using the fact that $(V_k-V)^2 = V_k^2 - V^2 + 2V(V-V_k)$ and
\[
\int_{B_r(x)} V_k^2-V^2 dx + \int_{B_r(x)} 2V(V-V_k) dx \to 0 \text{ as } k \to \infty
\]
we conclude $V_k \to V$ strongly in $L^2_{loc}(B_1)$.
\end{proof}

We now use the energy estimates of Proposition \ref{thirdderivs} to prove a small energy compactness result.

\begin{lem}\label{strong2}Let $u_k$ be a sequence of $f_k$-approximate biharmonic maps in $B_2$ with $f_k \in L\log L(B_2)$ satisfying \eqref{globalest}. There exists $\varepsilon_0>0$ such that if $\|Du_k\|_{L^4(B_2)} + \|D^2u_k\|_{L^2(B_2)} < \varepsilon_0$, then there exists $u \in W^{2,2}_{loc}(B_{2})$ such that 
\[
Du_k \to Du \text{ strongly in }L^4_{loc}(B_1),
\]
\[
D^2u_k \to D^2u\text{ strongly in }L^2_{loc}(B_1).
\]
\end{lem}
\begin{proof}We will first prove convergence of $Du_k$ to $Du$ and $D^2u_k$ to $D^2 u$ in $L^2_{loc}$ and then use Gagliardo-Nirenberg interpolation to get the $L^4$ convergence. 

Begin by choosing $0<\varepsilon_0 < \tilde \varepsilon$ from Proposition \ref{thirdderivs}. First note that the uniform bounds on $u_k$ in $W^{2,2}(B_2)$ implies that there exists a $u \in W^{2,2}_{loc}(B_2)$ such that $u_k \rightharpoonup u$ in $W^{2,2}_{loc}(B_2)$. We now show the strong convergence for the derivatives indicated.

Pick any $x_0 \in B_1$ and $2R \in (0, {\frac 34}]$. Then $B_{2R}(x_0) \subset B_2$. Let $\hat u_k(x):= u_k(x_0+ 2Rx), \hat f_k(x):= (2R)^4f_k(x_0+2Rx)$. Then $\hat u_k$ is an $\hat f_k$-approximate biharmonic map on $B_1$. 
From \eqref{epsreg2}, \eqref{epsreg3} we note that for any $r \in (0,1/2]$,
\begin{align*}
\|D\hat u_k\|_{L^4(B_r)}+&\|D^2 \hat u_k\|_{L^2(B_r)} \\
&\leq Cr(\|D\hat u_k\|_{L^4(B_2)}+\|D^2\hat u_k\|_{L^2(B_2)}) + C\left(\|D\hat u_k\|_{L^4(B_2)}^2+ \|D^2\hat u_k\|_{L^2(B_2)}^2 \right.\\
& \left.+( \|\hat f_k\|_{L^1(B_2)}\|\hat f_k\|_{L\log L(B_2)})^{\frac 12}+(\|\hat f_k\|_{L^1(B_2)}^3\|\hat f_k\|_{L\log L(B_2)})^{\frac 14} \right).
\end{align*}Using the scaling relations listed in Appendix \ref{scalingappend} and Lemma \ref{fLlogLlemma} we observe that
\begin{align*}
\|D u_k&\|_{L^4(B_{r2R}(x_0))}+\|D^2 u_k\|_{L^2(B_{r2R}(x_0))}\\
& \leq Cr(\|D u_k\|_{L^4(B_{2R}(x_0))}+\|D^2u_k\|_{L^2(B_{2R}(x_0))}) + C\left(\|D u_k\|_{L^4(B_{2R}(x_0))}^2+ \|D^2 u_k\|_{L^2(B_{2R}(x_0))}^2 \right.\\
& \left.+(\| f_k\|_{L^1(B_{2R}(x_0))}\|f_k\|_{L\log L(B_{2R}(x_0))})^{\frac 12}+( \|f_k\|_{L^1(B_{2R}(x_0))}^3\|f_k\|_{L\log L(B_{2R}(x_0))})^{\frac 14} \right).
\end{align*}
Lemma \ref{fL1lemma} and \eqref{globalest} together imply that
\[
\| f_k\|_{L^1(B_{2R}(x_0))} \leq C \left(\log\left(\frac 1{2R}\right)\right)^{-1}\|f_k\|_{L\log L(B_{2R}(x_0))}\leq C\Lambda \left(\log\left(\frac 1{2R}\right)\right)^{-1}.
\]Note that the right hand side goes to zero as $R \to 0$. Therefore, the small energy hypothesis implies that
\begin{align*}
\lim_{\substack{R \to 0}}\lim_{\substack{r \to 0}}&\lim_{\substack{k \to \infty}}\left(\|D u_k\|_{L^4(B_{r2R}(x_0))}+\|D^2 u_k\|_{L^2(B_{r2R}(x_0))}\right)\\
&\leq C\varepsilon_0\lim_{\substack{R \to 0}}\lim_{\substack{r \to 0}}\lim_{\substack{k \to \infty}}\left(\|D u_k\|_{L^4(B_{2R}(x_0))}+\|D^2 u_k\|_{L^2(B_{2R}(x_0))}\right).
\end{align*}Decreasing $\varepsilon_0$, if necessary, so that $\varepsilon_0< 1/C$, implies that
\[
\lim_{\substack{r \to 0}}\lim_{\substack{k \to \infty}}\left(\|D u_k\|_{L^4(B_{r}(x_0))}+\|D^2 u_k\|_{L^2(B_{r}(x_0))}\right)=0
\]for all $x_0 \in B_1$. Let $V_k = D^2 u_k$ and $V = D^2u$. Since $V_k \rightharpoonup V$ weakly in $L^2$ as measures and $V_k$ satisfies the hypotheses of Lemma \ref{measures} and Corollary \ref{measurecor} on $B_1$, $V_k \to V$ strongly in $L^2_{loc}(B_1)$.

Since $Du_k \rightharpoonup Du$ weakly as measures in $L^2(B_2)$ and 
\[
\lim_{\substack{r \to 0}}\lim_{\substack{k \to \infty}}\|D u_k\|_{L^2(B_{r}(x_0))} \leq \lim_{\substack{r \to 0}}\lim_{\substack{k \to \infty}}r\|D u_k\|_{L^4(B_{r}(x_0))}=0,
\]for all $x \in B_1$, Corollary \ref{measurecor} again implies that $Du_k \to Du$ strongly in $L^2_{loc}(B_1)$.

Now, for any $B_r(x) \subset B_1$, we consider the functions $w_k:= (u_k-u) -\, \mint_{B_r(x)} (u_k - u)$. Then, $Dw_k = D(u_k-u)$ and $D^2w_k = D^2(u_k-u)$. We apply the Gagliardo-Nirenberg interpolation inequality for $w_k$ and then the Poincar\'e inequality for the $L^2$ estimates on $w_k$ to conclude 
\[
\|Dw_k\|_{L^4(B_r(x))} \leq C\|D^2 w_k\|_{L^2(B_r(x))} \|Dw_k\|_{L^2(B_r(x))} + C \|Dw_k\|_{L^2(B_r(x))} .
\]
Then, using the strong convergence of $D^2u_k \to Du$ in $L^2_{loc}$ and $Du_k \to Du$ in $L^2_{loc}$ we conclude $Du_k \to Du$ in $L^4_{loc}(B_1)$.
\end{proof}

Finally, we prove the energy quantization result under the presumption of one bubble at the origin.

\begin{prop}\label{quantizeprop}
Let $f_k\in L\log L(B_1, \Real^{n+1})$ and $u_k \in W^{2,2}(B_1, \Ss^n)$ be a sequence of $f_k$-approximate biharmonic maps with bounded energy such that 
\[\begin{array}{ll}
u_k\to u &\text{in } W^{2,2}_{loc}(B_1\backslash \{0\}, \Ss^n)\\
\tilde u_k(x):= u_k(\lambda_k x)  \to \omega(x) &\text{in }W^{2,2}_{loc}(\mathbb R^4,\Ss^n).
\end{array} 
\]Presume further that $\omega$ is the only ``bubble'' at the origin.  Let $A_k(\delta, R):= \{ x | \lambda_k R \leq |x| \leq \delta\}$.
Then

\[
\lim_{R\to \infty} \lim_{\delta \to 0} \lim_{k \to \infty} \left( \|D^2 u_k\|_{L^2(A_k(\delta,R))}+ \|Du_k\|_{L^4(A_k(\delta,R))}+\|D\Delta u_k\|_{L^{\frac 43}(A_k(\delta,R))}\right)= 0.
\]
The proposition also holds if $u_k$ is a sequence of $f_k$-approximate intrinsic biharmonic maps.
\end{prop}
\begin{proof}We first prove that for any $\varepsilon >0$ there exists $K$ sufficiently large and $\delta$ small so that for all $k \geq K$ and $\rho_k >0$ such that $B_{4\rho_k} \backslash B_{\rho_k/2} \subset A_k(\delta, R)$
\begin{equation}\label{dyadicest}
\|D^2 u_k\|_{L^2(B_{2\rho_k}\backslash B_{\rho_k})}+ \|Du_k\|_{L^4(B_{2\rho_k}\backslash B_{\rho_k})}+\|D\Delta u_k\|_{L^{\frac 43}(B_{2\rho_k}\backslash B_{\rho_k})}<\varepsilon.
\end{equation}

Since $\{0\}$ is the only point of energy concentration, the strong convergence of $D^2u_k \to D^2u$ in $L^2$ and $Du_k \to Du$ in $L^4$ implies that for any $\varepsilon>0$ and any $m\in \mathbb Z^+$ and $\delta$ sufficiently small, there exists $K:=K(m)$ sufficiently large such that for all $k \geq K(m)$,
\begin{align}\label{smallannulus1}
\|D^2 u_k\|_{L^2({B_{2\delta} \backslash B_{\delta 2^{-m-1}}})} + \|Du_k\|_{L^4({B_{2\delta} \backslash B_{\delta 2^{-m-1}}})}\leq \frac \varepsilon {C\Gamma^{m+3}}.
\end{align}Here $C$ is an appropriately large constant determined by the bounds of Proposition \ref{thirdderivs} and $\Gamma$ is the number of balls of radius $r/32$ needed to cover $B_{r}\backslash B_{\frac r2}$. By \eqref{epsreg4}, for any $x \in B_{2\delta}\backslash B_{\delta 2^{-m-1}}$ and $0<r<\delta 2^{-m-1}$,
\begin{align}
\|D\Delta u_k&\|_{L^{\frac 43}({B_{\frac r{32}}(x)})}\leq C\left(\|D^2u_k\|_{L^{2}(B_{\frac r2}(x))}+\|Du_k\|_{L^4(B_{\frac r2}(x))}+ \|f_k\|_{L^1(B_{\frac r2}(x))}^{\frac 14}\|f_k\|_{L\log L(B_{\frac r2}(x))}^{\frac 34} \right)
%  \leq  C\r\| D^2u_k\|_{L^2(B_{\frac r2}(x))}\nonumber \\
%& +C\left(\|D^2u_k\|^{2}_{L^{2}(B_{\frac r2}(x))}+\|Du_k\|^{2}_{L^4(B_{\frac r2}(x))}+ \|f_k\|_{L^1(B_{\frac r2}%(x))}^{\frac 14}\|f_k\|_{L\log L(B_{\frac r2}(x))}\right).
\label{smallannulus2}
\end{align}
Since Lemma \ref{fL1lemma} and \eqref{globalest} imply that
\begin{equation}\label{smallannulus3}
\|f_k\|_{L^1(B_{\frac r2}(x))}\leq C \left(\log\left(\frac 1{r}\right)\right)^{-1} \|f_k\|_{L\log L(B_{\frac r2}(x))},
\end{equation}
for sufficiently small $\delta$, \eqref{smallannulus1}, \eqref{smallannulus2}, and \eqref{smallannulus3} together imply that for $k \geq K(m)$
\begin{equation}\label{smallnorm1}
\|D\Delta u_k\|_{L^{\frac 43}({B_{2\delta} \backslash B_{\delta 2^{-m-1}}})}+\|D u_k\|_{L^4({B_{2\delta} \backslash B_{\delta 2^{-m-1}}})}+\|D^2 u_k\|_{L^2({B_{2\delta} \backslash B_{\delta 2^{-m-1}}})}\leq \varepsilon/2.
\end{equation}
A similar argument (perhaps requiring a larger $K$) implies that 
\begin{equation}\label{smallnorm2}
\|D\Delta u_k\|_{L^{\frac 43}({B_{2^m \lambda_k R} \backslash B_{\lambda_k R}})}+\|D u_k\|_{L^{4}({B_{2^m \lambda_k R} \backslash B_{\lambda_k R}})}+\|D^2 u_k\|_{L^{2}({B_{2^m \lambda_k R} \backslash B_{\lambda_k R}})} \leq \varepsilon/2.
\end{equation}
Now suppose there exists a sequence $t_k$ with $\lambda_k R < t_k < \delta$ such that
\[
\|D^2 u_k\|_{L^2(B_{2t_k}\backslash B_{t_k})}+ \|Du_k\|_{L^4(B_{2t_k}\backslash B_{t_k})}+\|D\Delta u_k\|_{L^{\frac 43}(B_{2t_k}\backslash B_{t_k})}\geq \varepsilon.
\]By \eqref{smallnorm1} and \eqref{smallnorm2}, $t_k \to 0$ and $B_{\delta/t_k} \backslash B_{\lambda_k R/t_k} \to \Real^4 \backslash \{0\}$. Define $v_k(x) = u_k(t_kx)$ and $\tilde f_k(x):= t_k^4 f_k( t_kx)$. Then $v_k$ is an $\tilde f_k$-approximate biharmonic map, defined on $B_{t_k^{-1}}$. We first observe that $v_k \to v_\infty$ weakly in $W^{2,2}_{loc}(\Real^4,\Ss^n)$. Notice for any $R>0$
\begin{align*}
\int_{B_R} |\tilde f_k(x)|dx &= \int_{B_{Rt_k}} |f_k(s)| ds \\
&\leq \int_0^{|B_{Rt_k}|} (f_k)^*(t) dt \\
&\leq  c\left(\log\left(2+\frac 1{Rt_k}\right)\right)^{-1}\int_0^\infty (f_k)^*(t) \log(2+\frac 1t)dt\\
&=c \left(\log\left(2+\frac 1{Rt_k}\right)\right)^{-1}\|f_k\|_{L \log L(B_1)}.
\end{align*}By \eqref{globalest} $\tilde f_k \to 0$ in $L^1_{loc}(\Real^4)$. Moreover, for all $k$
\[
\|D^2v_k\|_{L^2(B_2\backslash B_1)} + \|Dv_k\|_{L^4(B_2 \backslash B_1)}+ \|D\Delta v_k\|_{L^{\frac 43}( B_2\backslash B_1)} \geq \varepsilon.
\] If $v_k \to v_\infty$ strongly in $W^{2,2}(B_{16} \backslash B_{1/16}, \Ss^n)$, then $v_\infty$ is a non-constant biharmonic map into $\Ss^n$. Note that by Proposition \ref{thirdderivs} we get 
\[
\|D^2 v_\infty\|_{L^2(B_2 \backslash B_1)} + \|Dv_\infty\|_{L^4(B_2\backslash B_1)} > 0.
\]This contradicts the fact that there is only one bubble at $\{0\}$. If the convergence is not strong, then Lemma \ref{strong2} implies that the energy must concentrate. That is, there exists a subsequence $v_k$ such that $\|D^2 v_k\|_{L^2(B_r(x))} + \|Dv_k\|_{L^4(B_r(x))} \geq \varepsilon_0^2$ for all $r>0$. This also contradicts the existence of only one bubble. Thus, \eqref{dyadicest} holds.

Using the duality of Lorentz spaces and the estimates of Appendix \ref{estap} we get the bounds
\begin{align}
\|D^2 u_k\|_{L^2}^2 &\leq C\|D^2u_k\|_{L^{2,\infty}}\|D^2 u_k\|_{L^{2,1}},\label{duallor}\\
\|Du_k\|_{L^4}^4 &\leq C\|\, |Du_k|^3\|_{L^{{\frac 43},\infty}}\|Du_k\|_{L^{4,1}}\leq C\|Du_k\|^3_{L^{4,\infty}}\|Du_k\|_{L^{4,1}},\nonumber\\
\|D\Delta u_k\|_{L^{{\frac 43}}}^{{\frac 43}}& \leq C\|(D\Delta u_k)^{1/3}\|_{L^{4,\infty}}\|D\Delta u_k\|_{L^{{\frac 43},1}} \leq C \|D\Delta u_k\|^{\frac 13}_{L^{{\frac 43},\infty}}\|D\Delta u_k\|_{L^{{\frac 43},1}}. \nonumber
\end{align}

Using \eqref{globalest} and \eqref{epsreg1} we observe that
\[
\|D^2 u_k\|_{L^{2,1}}+ \|Du_k\|_{L^{4,1}}+\|D\Delta u_k\|_{L^{{\frac 43},1}} \leq C\Lambda.
\]Since \eqref{dyadicest} allows us to apply Lemma \ref{inftylemma}, appealing to \eqref{duallor} implies the result.
\end{proof}

The full proof of Theorem \ref{bubblethm} now follows immediately from the uniform energy bounds of \eqref{globalest}, the small energy compactness results of this section, and standard induction arguments on the bubbles.
\section{Oscillation Bounds}\label{oscillation_section}
 The proof of the following oscillation lemma will constitute the work of this section.

\begin{lem}\label{osclemma}
Let $u\in W^{2,2}(B_1, \Ss^n)$ be an $f$-approximate biharmonic map for $f \in L\log L(B_1, \Real^{n+1})$ with
\[
\|D^2u\|_{L^2(B_1)}+\|Du\|_{L^4(B_1)}+ \|f\|_{L \log L(B_1)} \leq \Lambda < \infty.
\] Then for $0<2t<\delta/2<1/16$, 
\begin{align*}
\sup_{x,y \in B_{\delta/2} \backslash B_{2t}} |u(x)-u(y)| \leq &C\left( \|D^2u\|_{L^2(B_{2\delta} \backslash B_t)} + \|Du\|_{L^4(B_{2\delta} \backslash B_t)} + \|f\|_{L\log L(B_{2\delta})}\right.\\
&\qquad + \left. \|D \Delta u\|_{L^{\frac 43}(B_{2\delta} \backslash B_{t})}+\|D \Delta u\|_{L^{\frac 43,1}(B_{2t} \backslash B_{t})}+|B_{4\delta}|\right).
\end{align*}
The lemma also holds if $u$ is an $f$-approximate intrinsic biharmonic map.
\end{lem}
%\noindent The proof of this lemma will constitute the work of this entire section. 

Consider the map $u_1:B_1 \to \Real^{n+1}$ such that $u_1(\xX) = \mathbf b + A\xX$ where $\mathbf b \in \Real^{n+1}$ and $A$ is an $(n+1)\times 4$ matrix with
\[
A:= \mint_{B_{2t}\backslash B_t} Du, \; \; \mathbf b:= \mint_{B_{2t}\backslash B_t} (u(\xX) - A\xX) dVol(\xX).
\]Then by construction
\[
\mint_{B_{2t}\backslash B_t} u - u_1=0, \; \; \; \mint_{B_{2t}\backslash B_t} Du- Du_1=0, \; \; \; D^k u_1 \equiv 0 \text{ for all } k \geq 2.
\]Set $w = (1-\phi_t)(u-u_1)$. 
Let $w_1:B_1 \to \Real^{n+1}$ such that $w_1(\xX) = \mathbf m + N\xX$ where
\[
N:= \mint_{B_{\delta}\backslash B_{\delta/2} }Dw, \; \; \mathbf m:= \mint_{B_{\delta}\backslash B_{\delta/2}} (w(\xX) - N\xX) dVol(\xX).
\]
Let $\tilde w =(w-w_1) \phi_{\delta/2}$ so $\tilde w= w-w_1$ on $B_{\delta/2}$ and the support of $\tilde w$ is contained in ${B_\delta}$.

By definition
\begin{align*}
\sup_{x,y \in B_{\delta/2} \backslash B_{2t}} |u(x)-u(y)| &= \sup_{x,y \in B_{\delta/2} \backslash B_{2t}} |w(x)-w(y)+u_1(x)-u_1(y)|\\
& =  \sup_{x,y \in B_{\delta/2} \backslash B_{2t}} \left| (\tilde w + u_1+w_1)(x) - (\tilde w + u_1+w_1)(y)\right|\\
& \leq 2 \sup_{x \in B_{\delta/2} \backslash B_{2t}} \left| \tilde w(x)-\tilde w(0) +(A+N)x\right|. 
\end{align*}
We first observe that outside of $B_{2t}$, $w=u-u_1$ so the definition of $N$ implies that
\[
A+N = A+\mint_{B_{\delta}\backslash B_{\delta/2} } Du - \mint_{B_{\delta}\backslash B_{\delta/2} } A = \mint_{B_{\delta}\backslash B_{\delta/2} } Du.
\] Thus, for $x \in B_{\delta/2}$, H\"older's inequality implies that
\[
|(A+N)x| \leq C \delta^{-3} \int_{\Bd} |Du| \leq C \|Du\|_{L^4(\Bd)}.
\]
% By construction $\tilde w =0$ outside of $B_{\delta}$. 
As before, let $G$ be the distribution in $\Real^4$ such that $\Delta^2 G = \delta_0$. Then $G(x) = C\log |x|$ and recall that $DG\in L^{4,\infty}(\Real^4)$. It is enough to show that
\begin{claim}
\[| \tilde w(x) - \mint_{\Real^4}\tilde w| \leq C \|D\Delta \tilde w\|_{L^{\frac 43,1}(\Real^4)}\]
\end{claim}
Since all of the above quantities are translation invariant, we may assume $x=0$. Then
\begin{align*}
\left| \tilde w(0)-\mint \tilde w\right|&= \left |\int_{\Real^4} \Delta^2 G(y)(\tilde w (y)-\mint \tilde w)dV(y)\right| \\&= \left |\int_{\Real^4} D G(y) D\Delta\tilde w (y) dV(y)\right| \\&\leq C \|DG\|_{L^{4,\infty}(\Real^4)} \|D\Delta \tilde w\|_{L^{\frac 43,1}(\Real^4)}.
\end{align*}
Using the definition of $\tilde w$,
\begin{align*}
\|D\Delta \tilde w\|_{L^{\frac 43,1}(\Real^4)} \leq &C\|\left(\delta^{-3}|w-w_1| + \delta^{-2}|D(w-w_1)| + \delta^{-1}|D^2w|\right)\|_{L^{\frac 43,1}(\Bd)} \\
& + C\|D\Delta w\|_{L^{\frac 43,1}(B_{\delta})}.
\end{align*}Interpolation techniques and Poincar\'e's inequality imply that
\begin{align*}
\|\delta^{-3} (w-w_1)\|_{L^{\frac 43,1}(\Bd)} &\leq C \|\delta^{-2} D(w-w_1)\|_{L^{\frac 43,1}(\Bd)}\\
&\leq C \|\delta^{-1} D^2w\|_{L^{\frac 43,1}(\Bd)}.
\end{align*}Moreover, the embedding theorems for Lorentz spaces imply that
\[
\|\delta^{-1} D^2w\|_{L^{\frac 43,1}(\Bd)}\leq C\|D^2 w\|_{L^2(\Bd)}.
\]
Therefore,
\begin{equation}\label{thirdwest}
\|D\Delta \tilde w\|_{L^{\frac 43,1}(\Real^4)}\leq C \|D^2 w\|_{L^2(\Bd)} + C\|D\Delta w\|_{L^{\frac 43,1}(B_{\delta})}.
\end{equation}
Since $D^2 w = D^2u$ on $B_\delta \backslash B_{2t}$, we conclude
\begin{equation}\label{oscest}
\osc_{ B_{\delta/2} \backslash B_{2t}} u\leq C\left(\|D\Delta w\|_{L^{\frac 43,1}(B_{\delta})} +  \|D^2 u\|_{L^2(\Bd)}+\|Du\|_{L^4(\Bd)}\right).
\end{equation}The remainder of the proof will be devoted to bounding the $D\Delta w$ term.

We define $\beta = D\Delta w \wedge u - \Delta w \wedge Du$. Then $\beta^{ij}:= u^jD\Delta w^i  -u^i D\Delta w^j - \Delta w^i Du^j + \Delta w^j Du^i \in \Omega^1\Real^4$ for $i,j = 1, \dots, n+1$. By definition $\beta = D\Delta u \wedge u - \Delta u \wedge Du$ in $B_\delta \backslash B_{2t}$ and thus $d^*\beta = f \wedge u$ in $B_\delta \backslash B_{2t}$. We will require an $L^{\frac 43}$ bound for $\beta$ and to that end note that
\begin{align}
\begin{split}
\label{betaest}\|\beta\|_{L^{\frac 43}(B_{2\delta})} & \leq C\left(\|D\Delta w\|_{L^{\frac 43}(B_{2\delta})} + \|\Delta w \wedge Du\|_{L^{\frac 43}(B_{2\delta})} \right)\\
&\leq C\left(\|D\Delta w\|_{L^{\frac 43}(B_{2\delta})}+ \|\Delta w\|_{L^2(B_{2\delta})}\|Du\|_{L^4(B_{2\delta})}\right) \\
&\leq C\left(\|D\Delta u\|_{L^{\frac 43}(B_{2\delta}\backslash B_{t})} + \|D^2 u\|_{L^2(B_{2\delta} \backslash B_t)} \right).
\end{split}
\end{align}For the last inequality, $\|Du\|_{L^4(B_{2\delta})}$ is bounded and is absorbed into the constant. In addition, we use the definition of $w$ and repeated applications of Poincar\'e and H\"older to determine
\[
\begin{array}{l}
\|D\Delta w\|_{L^{\frac 43}(B_{2\delta})} \leq C\left(\|D^2u\|_{L^2(\Bt)} + \|(1-\phi_t) D\Delta u\|_{L^{\frac 43}(B_{2\delta})} \right),\\
\|\Delta w\|_{L^2(B_{2\delta})} \leq C \|D^2 u\|_{L^2(B_{2\delta}\backslash B_t)}.
\end{array}
\]
 Set 
\[\gamma:=d^*\left(D\Delta(w-u)\wedge u - \Delta (w-u)\wedge Du\right).
\]Then
\[
d^* \beta = f \wedge u + \gamma; \qquad \qquad d\beta =-2D\Delta w \wedge Du,
\]
\[
\Delta \beta = (dd^* + d^*d) \beta = d(f \wedge u + \gamma) + d^*(-2D\Delta w \wedge Du).
\]We consider a decomposition for each component $\beta^{ij}$ such that $\beta^{ij} = H^{ij} + d\Psi^{ij} + d^* \Phi^{ij}$ where $H^{ij}$ is a harmonic one-form and $\Phi, \Psi$ satisfy appropriate partial differential equations. Our objective is to bound $\|D\Delta w\|_{L^{\frac 43,1}}$ by $\|\beta\|_{L^{\frac 43,1}}$ and to that end we determine such bounds for $d\Psi, d^* \Phi, H$.

\begin{rem}
For the intrinsic case, we modify a few definitions. Let $\beta_I := \beta + 2|Du|^2 Dw_I \wedge u$ where $w_I = (1-\phi_t)(u - \mathbf d)$ and $\mathbf d:= \mint_{\Bt} u$. Using the definition of $w_I$, we get the bound $\|\beta_I\|_{L^{\frac 43}(B_{2\delta})}\leq \|\beta\|_{L^{\frac 43}(B_{2\delta})}+ C\|Du\|_{L^4(B_{2\delta}\backslash B_t)}$ by using H\"older's inequality and Poincar\'e's inequality. We then define $\gamma_I:= \gamma + d^*(2|Du|^2 D(w_I-u) \wedge u)$ and thus 
\[
d^*\beta_I = f \wedge u + \gamma_I; \qquad \qquad d\beta_I = d\beta + D(|Du|^2)Dw_I \wedge u - |Du|^2 Dw_I \wedge Du.
\]
\end{rem}
We now continue with the proof for the extrinsic case.
 \begin{prop}\label{psiest}Let $\Psi^{ij}$ be a function on $B_{2\delta}$ satisfying 
\[
\left\{\begin{array}{ll} \Delta \Psi^{ij} = f^iu^j-f^ju^i+\gamma^{ij}& \text{in } B_{2\delta}\\
\Psi^{ij} = 0 & \text{on } \partial B_{2\delta}.
\end{array}\right.
\]Then
\begin{align*}
\|d \Psi^{ij}\|_{L^{\frac 43,1}(B_{2\delta})} \leq& C\left(\|D^2u\|_{L^2(\Bt)} + \|Du\|_{L^4(\Bt)} + \|D\Delta u\|_{L^{\frac 43}(\Bt)}\right. \\
& \qquad\left. + \|f\|_{L\log L(B_{2\delta})}+|B_{4\delta}|\right).
\end{align*}
\end{prop}
\begin{proof}We separate $\Psi^{ij} = \Psi^{ij}_1 + \Psi^{ij}_2$ so that 
\[
\left\{\begin{array}{ll} \Delta \Psi^{ij}_1 = \gamma^{ij}& \text{in } B_{2\delta}\\
\Psi^{ij}_1 = 0 & \text{on } \partial B_{2\delta}.
\end{array}\right.
\]

 Following classical arguments 
 \[
 \|D^2 \Psi_1^{ij}\|_{L^1(B_{2\delta})} \leq C \|\gamma^{ij}\|_{\mathcal H^1(B_{2\delta})}.
 \]
Thus the embedding theorems imply that $\|D\Psi_1^{ij}\|_{L^{\frac 43,1}(B_{2\delta})} \leq C \|\gamma^{ij}\|_{\mathcal H^1(B_{2\delta})}$. Now we consider the $\mathcal H^1$ norm of $\gamma^{ij}$. 
By definition
\begin{align*}\gamma^{ij}&= d^*\left(D\Delta (w^i-u^i)u^j-D\Delta (w^j-u^j)  u^i -\left[\Delta (w^i-u^i)Du^j-\Delta (w^j-u^j)Du^i \right]\right)\\
&=\Delta^2(w^i-u^i)u^j- \Delta^2(w^j-u^j)u^i  -( \Delta(w^i-u^i)\Delta u^j- \Delta(w^j-u^j)\Delta u^i). 
\end{align*}Recall that $w := (1-\phi_t)(u -u_1)$. So
\begin{align*}
\Delta(w^j-u^j) &= -\Delta \phi_t(u^j-u_1^j) -2D\phi_t\cdot D(u^j-u_1^j) - \phi_t \Delta u^j\\
\Delta^2(w^j-u^j)&=-\Delta^2 \phi_t (u^j-u_1^j) - \Delta\phi_t \Delta u^j  - 2 D\Delta \phi_t D(u^j-u_1^j) \\ &\qquad-2 \Delta(D\phi_t \cdot D(u^j-u_1^j))
 - \Delta \phi_t \Delta u^j -2 D\phi_t D\Delta u^j - \phi_t \Delta^2 u^j.
\end{align*}

Combining all of the terms we estimate
\begin{align*}
|\gamma^{ij}| \leq&C |D^4\phi_t|\,|u-u_1|+ C|D^3 \phi_t|\,|D(u-u_1)|+C|D^2 \phi_t|\,|D^2 u| \\
&+ C |D\phi_t| \left(|D\Delta u|+ |D(u-u_1)|\,|\Delta u|\right) + |\phi_t|\left| u^i\Delta^2 u^j-u^j\Delta^2 u^i\right| 
\end{align*}The definition of $\gamma^{ij}$ implies that 
$\gamma^{ij} = 0$ on $\Real^4 \backslash B_{2t}$ and 
\[\int_{\Real^4} \gamma^{ij}= \int_{\partial B_{2t}} \left(D\Delta(w-u)\wedge u - \Delta (w-u)\wedge Du\right)^{ij}\cdot \mathbf n =0.\]
The estimate from Lemma \ref{lphardy} implies that
\begin{align*}
\|\gamma^{ij}\|_{\mathcal H^1(B_{2\delta})} \leq &c\left(t\|\gamma^{ij}-\phi_t(u^j\Delta^2 u^i-u^i\Delta^2 u^j)\|_{L^{\frac 43}(B_{2t})} \right.\\
&\left.\qquad+ \|\phi_t(u^j\Delta^2 u^i-u^i\Delta^2 u^j)\|_{L \log L(B_{2\delta})}+|B_{4\delta}|\right).
\end{align*}
Repeating techniques used previously, we bound the first three terms of $|\gamma^{ij}|$:
\begin{align*}
t^{-4}\|u-u_1\|_{L^{\frac 43}(\Bt)} &\leq Ct^{-3}\|D( u-u_1)\|_{L^{\frac 43}(\Bt)}\\& \leq Ct^{-2}\|D^2 u\|_{L^{\frac 43}(\Bt)} \leq Ct^{-1}\|D^2 u\|_{L^2(\Bt)}.
\end{align*}
We will preserve the term
\[
t^{-1}\|D\Delta u\|_{L^{\frac 43}(\Bt)}
\]as our energy quantization result implies that this term will vanish when taking limits. H\"older's inequality and the fact that $ \|D(u-u_1)\|_{L^{4}(\Bt)} \leq C\|Du\|_{L^4(\Bt)}$ implies that
\begin{align*}
\|D(u-u_1)\Delta u\|_{L^{\frac 43}(\Bt)}&\leq C \|D(u-u_1)\|_{L^{4}(\Bt)}\|D^2 u\|_{L^{2}(\Bt)}\\
& \leq C \|D^2 u\|_{L^2(\Bt)}.
\end{align*}For the last term, since $u$ is an $f$-approximate biharmonic map into $\Ss^n$,
\[
\|\phi_t(\Delta^2 u \wedge u)\|_{L\log L(B_{2\delta})} \leq  \|f\wedge u\|_{L\log L(B_{2\delta})} \leq  \|f\|_{L\log L(B_{2\delta})}.
\]
All of the above estimates imply that
\begin{align*}
\|\gamma^{ij}\|_{\mathcal H^1(B_{2\delta})} \leq C\left(\right.&\left.\|D^2u\|_{L^2(\Bt)} + \|Du\|_{L^4(\Bt)} + \|D\Delta u\|_{L^{\frac 43}(\Bt)} \right. \\ & \qquad +  \left. \|f\|_{L\log L(B_{2\delta})}+|B_{4\delta}|\right).
\end{align*}
Finally consider 
\[
\left\{\begin{array}{ll} \Delta \Psi^{ij}_2 = f^iu^j-u^if^j& \text{in } B_{2\delta}\\
\Psi^{ij}_2 = 0 & \text{on } \partial B_{2\delta}.
\end{array}\right.
\]Then classical results give $\|\Psi^{ij}_2\|_{W^{2,1}(B_{2\delta})} \leq C \|f\|_{\mathcal H^1(B_{2\delta})} \leq C\|f\|_{L \log L(B_{2\delta})}$. Thus
\[
\|d\Psi^{ij}_2\|_{W^{1,1}(B_{2\delta})} \leq C\|f\|_{L \log L(B_{2\delta})}
\] and the embedding theorems in $\Real^4$ imply that
\[
\|d\Psi^{ij}_2\|_{L^{\frac 43,1}(B_{2\delta})} \leq C\|f\|_{L \log L(B_{2\delta})}.
\]
\end{proof}
\begin{rem}
For the intrinsic case, we define
\begin{align*}\gamma_I& = \gamma + d^*(2|Du|^2 D(w_I-u) \wedge u)\\
& = \gamma -2\phi_t d^*(|Du|^2Du\wedge u)+ 2|Du|^2 \left( \Delta \phi_t (\mathbf d-u) \wedge u-D\phi_t\cdot Du \wedge (\mathbf d +u)\right)\\
& \qquad + 2 D|Du|^2 \cdot D\phi_t (\mathbf d-u) \wedge u.
\end{align*} We bound $\|\gamma_I\|_{\mathcal H^1}$ by making the following observations. First, $-2\phi_t d^*(|Du|^2Du\wedge u)$ is added to the term $-\phi_t\Delta^2 u \wedge u$ that appears in the expansion of $\gamma$. We then make the substitution $-\phi_t f \wedge u$ as in the extrinsic case. Second, using Poincar\'e's inequality, H\"older's inequality, and the global energy bound for $u$, the $L^{\frac 43}$ norm of what remains is bounded by
$Ct^{-1}\left(\|Du\|_{L^4(\Bt)}+ \|D^2u\|_{L^2(\Bt)}\right)$. Finally, observe that by construction, $\gamma_I$ is supported on $B_{2t}$ and $\int_{\Real^4} \gamma_I = 0$ so the estimate used for $\|\gamma\|_{\mathcal H^1}$ still applies.
\end{rem}
\begin{prop}\label{phiest}
Let $\Phi^{ij}\in \Omega^2\Real^4$ be the solution to the system
\[
\left\{\begin{array}{ll} \Delta \Phi^{ij}= -2 (D\Delta w^i Du^j-D\Delta w^j Du^i)& \text{in } B_{2\delta}\\
\Phi^{ij} = 0 & \text{on } \partial B_{2\delta}.
\end{array}\right.
\]Then
\begin{equation}
\|d^* \Phi^{ij}\|_{L^{\frac 43,1}(B_{2\delta})} \leq  C\left(\|D^2 u\|_{L^2(\Bt)} + \|D\Delta u\|_{L^{\frac 43}(B_{2\delta} \backslash B_{t})}\right).
\end{equation}
\end{prop} 
\begin{proof}Using the same techniques and estimates as in the previous proposition we note that
\begin{align*}
\|d\Phi^{ij}\|_{L^{\frac 43,1}(B_{2\delta})} &\leq C\|D\Delta w \wedge Du\|_{\mathcal H^1(B_{2\delta})}\\
&\leq C\|D\Delta w\|_{L^{\frac 43}(B_{2\delta})}\|Du\|_{L^4(B_{2\delta})}\\
&\leq C\left(\|D^2 u\|_{L^2(\Bt)} + \|D\Delta u\|_{L^{\frac 43}(B_{2\delta} \backslash B_{t})}\right).
\end{align*}
\end{proof}
\begin{rem}
In the intrinsic setting the steps of the proof are the same, though the equation for $\Delta \Phi^{ij}_I$ includes the terms $D(|Du|^2)Dw_I \wedge u - |Du|^2 Dw_I \wedge Du$. Since $\|Dw_I\|_{L^4(B_{2\delta})} \leq C\|Du\|_{L^4(B_{2\delta} \backslash B_{t})}$, one can quickly show the intrinsic bound has the form
\[\|d^*\Phi_I\|_{L^{\frac 43,1}(B_{2\delta})}\leq \|d^*\Phi\|_{L^{\frac 43,1}(B_{2\delta})}+ C\|Du\|_{L^4(B_{2\delta} \backslash B_{t})}.
\]
\end{rem}
Now consider the harmonic one form 
\[
H^{ij} = \beta^{ij} -  d^* \Phi^{ij} - d\Psi^{ij}.
\]Propositions \ref{psiest} and \ref{phiest}, along with \eqref{betaest} imply that
\begin{align*}
\|H\|_{L^{\frac 43}(B_{2\delta})} & \leq \|\beta\|_{L^{\frac 43}(B_{2\delta})} + \|d^*\Phi\|_{L^{\frac 43}(B_{2\delta})} + \|d\Psi\|_{L^{\frac 43}(B_{2\delta})}\\
& \leq C\left( \|D^2u\|_{L^2(B_{2\delta} \backslash B_t)} + \|Du\|_{L^4(B_{2\delta} \backslash B_t)} \right.\\
& \qquad \left.+ \|D \Delta u\|_{L^{\frac 43}(B_{2\delta} \backslash B_{t})}+\|f\|_{L \log L(B_{2\delta})}+|B_{4\delta}|\right).
\end{align*}The mean value property and H\"older's inequality together imply that
\begin{align*}
\|H^{ij}\|_{C^0(B_{\delta})} &\leq \frac{C}{\delta^3}\left(  \|D^2u\|_{L^2(B_{2\delta} \backslash B_t)} + \|Du\|_{L^4(B_{2\delta} \backslash B_t)} \right. \\& \qquad \left.+ \|D \Delta u\|_{L^{\frac 43}(B_{2\delta} \backslash B_{t})}+\|f\|_{L \log L(B_{2\delta})}+|B_{4\delta}|\right).
\end{align*}
Moreover, a straightforward calculation implies that
\[
\|H^{ij}\|_{L^{\frac 43,1}(B_\delta)} \leq C\delta^{3}\|H^{ij}\|_{C^0(B_\delta)}.
\]
Thus, 
\begin{align*}
\|\beta\|_{L^{\frac 43,1}(B_\delta)} &\leq C\left(  \|D^2u\|_{L^2(B_{2\delta} \backslash B_t)} + \|Du\|_{L^4(B_{2\delta} \backslash B_t)} \right.\\
& \qquad \left.+ \|D \Delta u\|_{L^{\frac 43}(B_{2\delta} \backslash B_{t})}+\|f\|_{L \log L(B_{2\delta})}+|B_{4\delta}|\right).
\end{align*}Using the appropriate harmonic one form $H_I$, we produce the identical estimate for $\beta_I$.

We now use the definitions of $w$ and $\beta$ to determine a bound on $\|D\Delta w\|_{L^{\frac 43,1}(B_\delta)}$. First we consider the function on $B_{2t}$:
\begin{align*}
\|D\Delta w\|_{L^{\frac 43,1}(B_{2t})}&\leq \|C\left( t^{-3}|u-u_1| + t^{-2}|D(u-u_1)| + t^{-1}|D^2u|\right)\|_{L^{\frac 43,1}({\Bt})} \\& \qquad
+\| (1-\phi_t)D\Delta u\|_{L^{\frac 43,1}(B_{2t})}\\
&\leq C\|D^2 u\|_{L^2(\Bt)} + C\|D\Delta u\|_{L^{\frac 43,1}(\Bt)}.
\end{align*}

On $B_{\delta} \backslash B_{2t}$, $w=u-u_1$ so $D\Delta w \equiv D\Delta u$. We first decompose $D\Delta u$ into tangential and normal parts with tangency relative to the target manifold $\Ss^n$. Then
\[
D\Delta u = D\Delta u^T + D\Delta u^N = D\Delta u \wedge u.u + \langle D\Delta u, u\rangle u.
\]Here we define $\langle Dv, u\rangle:= \sum_{i,k} \frac{\partial v^k}{\partial x_i} u^k dx_i$. 
On $B_{\delta}\backslash B_{2t}$, $D\Delta u \wedge u = \beta + \Delta u \wedge Du$ and thus
\[
|(D\Delta u)^T| \leq |\beta| + |\Delta u|\,|Du|.
\]Since
\begin{align*}
\langle D\Delta u, u\rangle &= D\langle \Delta u, u\rangle -\langle\Delta u, Du\rangle\\ &= D( d^*\langle Du,u\rangle-|Du|^2)-\langle \Delta u, Du\rangle = -D|Du|^2 - \langle\Delta u, Du\rangle
\end{align*}we estimate

\begin{align*}
\|D\Delta w\|_{L^{\frac 43,1}(B_{\delta}\backslash B_{2t})} &\leq C\|\beta\|_{L^{\frac 43,1}(B_{\delta})} +C \|D^2 u\|_{L^2(B_{\delta}\backslash B_{2t})} \|Du\|_{L^4(B_{\delta}\backslash B_{2t})}\\
&\leq C\left(  \|D^2u\|_{L^2(B_{2\delta} \backslash B_t)} + \|Du\|_{L^4(B_{2\delta} \backslash B_t)} + \|D \Delta u\|_{L^{\frac 43}(B_{2\delta} \backslash B_{t})}\right.\\
&\qquad + \left. \|f\|_{L \log L(B_{2\delta})}+|B_{4\delta}|\right).%+\|D^2 u\|_{L^2(B_{2\delta}\backslash B_{2t})} \|Du\|_{L^4(B_{2\delta}\backslash B_{2t})}\right).
\end{align*}
Thus, 
\begin{align*}
\|D\Delta w\|_{L^{\frac 43,1}(B_\delta)} \leq &C\left( \|D^2u\|_{L^2(B_{2\delta} \backslash B_t)} + \|Du\|_{L^4(B_{2\delta} \backslash B_t)} +\|f\|_{L \log L(B_{2\delta})} \right.\\
&\qquad \left. +\|D \Delta u\|_{L^{\frac 43,1}(B_{2t} \backslash B_{t})} +\|D \Delta u\|_{L^{\frac 43}(B_{2\delta} \backslash B_{t})}+|B_{4\delta}| \right).
\end{align*}
Inserting this inequality into \eqref{oscest} proves the oscillation lemma.
\begin{rem}
To complete the proof in the intrinsic case, observe that on $B_{\delta} \backslash B_{2t}$, $D\Delta w \wedge u = D\Delta u \wedge u = \beta + \Delta u \wedge Du+ 2|Du|^2 Du \wedge u$. This changes the $L^\infty$ estimate for $|(D\Delta u)^T|$ on $B_{\delta} \backslash B_{2t}$ but using embedding theorems for Lorentz spaces we note that the $L^{\frac 43,1}$ estimate is unchanged.
\end{rem}
\section{No neck property -- Proof of Theorem \ref{neckthm}}\label{no_neck_section}

The proof of the no neck property now follows easily from combining the energy quantization and the oscillation bounds.

\begin{proof}As we may use induction to deal with the case of multiple bubbles, we prove the theorem for one bubble. 
Let $\lambda_k$ be such that $\tilde u_k(x):= u_k(\lambda_k x) \to \omega(x) \in W^{2,2}_{loc}(\Real^4, \Ss^n)$. Since each of the $u_k \in W^{2,2}(B_1, \Ss^n)$ are $f_k$-approximate biharmonic maps with $f_k \in L \log L(B_1, \Real^{n+1})$ and have uniform energy bounds, Lemma \ref{osclemma} implies that
\begin{align*}
\sup_{x,y \in B_{\delta/2}\backslash B_{2\lambda_kR}}\left|u_k(x)-u_k(y)\right| & \leq C\left( \|D^2u_k\|_{L^2(B_{2\delta} \backslash B_{{\lambda_kR}/{2}})}+\|Du_k\|_{L^4(B_{2\delta} \backslash B_{{\lambda_kR}/{2}})} \right. \\ & \qquad \left. +\|f_k\|_{L \log L(B_{2\delta})}+ \|D \Delta u_k\|_{L^{\frac 43,1}(B_{\lambda_kR} \backslash B_{{\lambda_kR}/{2}})}\right.\\
& \qquad \left. + \|D \Delta u_k\|_{L^{\frac 43}(B_{2\delta} \backslash B_{{\lambda_kR/2}})}+|B_{4\delta}|\right).
\end{align*}

%\begin{align*}
%\lim_{R \to \infty} \lim_{\delta \to 0} \lim_{k \to \infty} &\left( \|D^2u_k\|_{L^2(B_{2\delta} \backslash B_{{\lambda_kR}/{2}})}+\|Du_k\|_{L^4(B_{2\delta} \backslash B_{{\lambda_kR}/{2}})} \right. \\ & \qquad \left. + \|D \Delta u_k\|_{L^{\frac 43,1}(B_{2\delta} \backslash B_{{\lambda_kR}/{2}})}+\|f_k\|_{L \log L(B_{2\delta})}\right)=0.\end{align*}
Theorem \ref{bubblethm} implies that
\[
\lim_{\substack{\delta\to 0}}\lim_{\substack{R\to \infty}}\lim_{\substack{k\to \infty}}\left(\|D^2u_k\|_{L^2(B_{2\delta} \backslash B_{{\lambda_kR}/{2}})}+\|Du_k\|_{L^4(B_{2\delta} \backslash B_{{\lambda_kR}/{2}})} + \|D \Delta u_k\|_{L^{\frac 43}(B_{2\delta} \backslash B_{{\lambda_kR/2}})}\right)=0.
\]
%Recall that there exists a continuous increasing function $h$ such that
%\[
%\|f\|_{L \log L} \leq h(\|f\|_{L \log L})
%\]and $h(0)=0$.
%Then,
%\[
%\lim_{\delta \to 0} \sup_k \|f_k\|_{L \log L(B_{2\delta})}=\lim_{\delta \to 0} \sup_k h(\|f_k\|_{L \log L(B_{2\delta})})=h(0)=0.
%\]
Further, \eqref{epsreg1} and H\"older's inequality imply that
\begin{align*}
\|D\Delta u_k\|_{L^{\frac 43,1}(B_{\lambda_kR} \backslash B_{{\lambda_kR}/{2}})} \leq C&\left( \|D u_k\|_{L^{4}((B_{2\lambda_kR} \backslash B_{{\lambda_kR}/{4}})}\right.\\
& \left.+\|D^2u_k\|_{L^{2}((B_{2\lambda_kR} \backslash B_{{\lambda_kR}/{4}})}+ \|f_k\|_{L\log L(B_{2\lambda_kR})}\right).
\end{align*}Since we presume the $L \log L$ norm of $f_k$ does not concentrate
\[
 \lim_{\substack{\delta\to 0}}\lim_{\substack{R\to \infty}}\lim_{\substack{k\to \infty}}\|f_k\|_{L \log L(B_{2\delta})} =0.
\]

Therefore, 
\[
 \lim_{\substack{\delta\to 0}}\lim_{\substack{R\to \infty}}\lim_{\substack{k\to \infty}}\|D \Delta u_k\|_{L^{\frac 43,1}(B_{\lambda_kR} \backslash B_{{\lambda_kR}/{2}})}=0.
\]Taking all of the estimates together implies that
\[
 \lim_{\substack{\delta\to 0}}\lim_{\substack{R\to \infty}}\lim_{\substack{k\to \infty}}\sup_{x,y \in B_{\delta/2}\backslash B_{2\lambda_kR}}\left|u_k(x)-u_k(y)\right| =0.
\]Thus, no neck occurs in the blowup.
\end{proof}
\begin{rem}
For $f_k \in \phi(L)$, we use the following estimate
\begin{align*}
\|f_k\|_{L \log L(B_{2\delta})} &= \int_{B_{2\delta} \cap \{|f_k|\leq \delta^{-1}\}} |f_k|\log(2+ |f_k|) dx + \int_{|f_k|> \delta^{-1}}|f_k|\log(2+ |f_k|) dx \\
&\leq C \delta^3 \log(2+ \delta^{-1}) + \sup_{\substack{t>\delta^{-1}}}\frac{t \log (2+t)}{\phi(t)} \int_{|f_k|> \delta^{-1}}\phi(|f_k|) dx\\
& \leq C\delta^3 \log(2+ \delta^{-1}) + \sup_{\substack{t>\delta^{-1}}}\frac{t \log (2+t)}{\phi(t)} \Lambda.
\end{align*}Since we presumed $\lim_{t \to \infty} \frac {\phi(t)}{t\log t} =\infty$ we determine
\[
\lim_{\delta \to 0} \sup_k \|f_k\|_{L \log L(B_{2\delta})}=0.
\]
\end{rem}

\appendix
\section{Necessary Background}
\subsection{Hardy Spaces, Lorentz Spaces, $L \log L$, and Orlicz Spaces}

Let $T:=\{ \Phi \in C^\infty(\Real^4) | \spt(\Phi) \subset B_1, \|\nabla \Phi\|_{L^\infty(\Real^4)} \leq 1\}$. For any $\Phi \in T$, let $\Phi_t(x):= t^{-4}\Phi\left(\frac xt\right)$. For each $f \in L^1(\Real^4)$, let
\[
f_*(x)= \sup_{\Phi \in T}\sup_{t>0}\left|(\Phi_t \ast f)(x)\right|.
\] Then $f$ is in the \emph{Hardy space} $\mathcal H^1(\Real^4)$ if $f_* \in L^1(\Real^4)$ and
\[
\|f\|_{\mathcal H^1(\Real^4)} = \|f_*\|_{L^1(\Real^4)}.
\]Thus, one has the continuous embedding $\mathcal H^1 \hookrightarrow L^1$.

For a measurable function $f:\Omega \to \Real$, let $f^*$ denote the non-increasing rearrangement of $|f|$ on $[0, |\Omega|)$ such that
\[
\left|\{x \in \Omega | \, |f(x)|\geq s\}\right| = \left| \{t \in (0, |\Omega|) | \, f^*(t) \geq s\}\right|.
\]
Let 
\[
f^{**}(t) := \frac 1t \int_0^t f^*(s) ds.
\]For $p \in (1, \infty)$, let
\begin{equation*}
\|f\|_{L^{p,q}}=\left\{ \begin{array}{ll}
\int_0^\infty t^{\frac 1p-1} f^{**}(t) dt,& \text{ if } q=1,\\
\sup_{t>0} t^{\frac 1p}f^{**}(t), & \text{ if } q = \infty.
\end{array}\right.
\end{equation*}We will also occasionally exploit the fact that one may understand $\|f\|_{L^{p, \infty}}$ by understanding instead its semi-norm
\[
\|f\|_{L^{p,\infty}}^*:= \sup_{\lambda>0}\lambda\left|\{x: |f(x)|> \lambda\}\right|^{\frac 1p}.
\]
We define the Banach spaces 
\[
L^{p,q}:= \{ f | \, \|f\|_{L^{p,q} } < \infty\}.
\]The spaces $L^{p,1}, L^{p,\infty}$ are examples of \emph{Lorentz spaces} and can be thought of as interpolation spaces between the standard $L^p$ spaces. For example, one observes that the following embeddings are all continuous:
\[
L^r(B_1) \hookrightarrow L^{p,1}(B_1) \hookrightarrow L^{p,p}(B_1)=L^p(B_1) \hookrightarrow L^{p,\infty}(B_1) \hookrightarrow L^q(B_1)
\]for all $q<p<r$, \cite{Helein}.

We define
\[
L \log L:= \{ f | \, \int |f(x)| \log (2 +|f(x)|)dx < \infty\}.
\]Since the above is non-linear, we will use the following semi-norm which is equivalent to the norm for $L\log L$
\[
\|f\|_{L \log L} := \int f^*(t) \log (2 + \frac 1t) dt.
\]
%Note there exists an increasing continuous function $h:[0, \infty) \to [0, \infty)$ with $h(0)=0$ such that 
%\[
%\|f\|_{L \log L} \leq h(\|f\|_{L \log L})
%\]for any $f \in L \log L$. 
We also note that $L^p(B_1) \hookrightarrow L \log L(B_1) \hookrightarrow L^1(B_1)$ are continuous embeddings for all $p>1$. Finally, we say $f$ is in $\mathcal H^1(B_1)$ if
\[
\left(f - \mint_{B_1} f(x)\,dx\right) \chi_{B_1} \in \mathcal H^1(\Real^4).
\]We record here the often used estimate
\begin{equation}
\|f\|_{\mathcal H^1(B_1)} \leq C \|f \|_{L \log L(B_1)}.
\end{equation}

Finally, for any increasing function $\phi:[0, \infty) \to [0, \infty)$ we define the Orlicz space
\[
\phi(L):= \{ f | \,\|f\|_{\phi(L)}:=\int \phi(|f|) dx < \infty\}. 
\]Examples include the $L^p$ spaces for $\phi(t)=t^p$ and $L \log L$ when $\phi(t) = t \log(2+t)$.

\subsection{Embeddings and Estimates for Lorentz spaces}\label{estap}
We will frequently use the following facts about Lorentz spaces:
\begin{enumerate}
\item $L^{p,q} \cdot L^{p',q'}$ continuously embeds into $L^{r,s}$ for $\frac 1p+ \frac 1{p'} \leq 1$ where
\[
\frac 1r = \frac 1p + \frac 1{p'}, \quad \frac 1s = \frac 1q + \frac 1{q'}
\]with
\[
\|fg\|_{L^{r,s}} \leq C\|f\|_{L^{p,q}} \|g\|_{L^{p',q'}}.
\]
\item For $f \in L^2, g \in W^{1,2}$
\[
\|fg\|_{L^{{\frac 43},1}} \leq C \|f\|_{L^2}\|g\|_{W^{1,2}}.
\]
\item \label{emb1} $W^{1,1}(\Real^4) \hookrightarrow L^{{\frac 43},1}(\Real^4)$ and $W^{1,2}(\Real^4) \hookrightarrow L^{4,2}(\Real^4)$ are continuous embeddings.
\item $L^{2,1}$ and $L^{2,\infty}$ are dual spaces, as are $L^{4,\infty}$, $L^{{\frac 43},1}$ and $L^{4,1}$, $L^{{\frac 43},\infty}$.
\item \label{Grafakos} For all $0<p,r<\infty$ and $0< q \leq \infty$, (see \cite{Grafakos}, Section 1.4.2)
\[\|f^r\|_{L^{p,q}}= \|f\|^r_{L^{pr,qr}}.\]

\item \label{Ziemer}Let $f \in L^{p,q}(\Real^4), g \in L^{p',q'}(\Real^4)$ with $\frac 1p + \frac 1{p'} >1$. Then $h= f \ast g \in L^{r,s}(\Real^4)$ where $\frac 1r = \frac 1p + \frac 1{p'} -1$ and $s$ is a number such that $\frac 1q + \frac 1{q'} \geq \frac 1s$. Moreover,
\[
\|h\|_{L^{r,s}(\Real^4)} \leq c \|f\|_{L^{p,q}(\Real^4)} \|g\|_{L^{p',q'}(\Real^4)}.
\]For a proof, see \cite{Ziemer}.

\vskip .1in
Let %$\Gamma$ be the distribution in $\Real^4$ such that $\Delta \Gamma = \delta_0$ and 
$G$ be the distribution such that $\Delta^2 G = \delta_0$. Then, $D^2G \in L^{2,\infty}(\Real^4)$ and $D^3G \in L^{{\frac 43},\infty}(\Real^4)$. Moreover, $D G \in L^{4,\infty}(\Real^4)$.

Using \eqref{Ziemer}, and considering $D^2G, D^3G$ as operators by convolution,

%\item $\Gamma: L^1(\Real^4) \to L^{2,\infty}(\Real^4)$, $D\Gamma: L^1(\Real^4)\to L^{{\frac 43},\infty}(\Real^4)$ are bounded operators.
\item $D^2 G:L^{{\frac 43},1}(\Real^4) \to L^{4,1}(\Real^4)$, $D^3G:L^{{\frac 43},1}(\Real^4) \to L^{2,1}(\Real^4)$ are bounded operators.

%\item We still need a reference for $DG: \mathcal H^1(B_1) \to L^{4,1}, D^2G: \mathcal H^1(B_1) \to L^{2,1}$
\end{enumerate}

\subsection{Scaling and estimates for $L\log L$ and $\mathcal H^1$}\label{scalingappend}
We first prove an essential but technical lemma that is probably well known, though we have not found a reference in the literature. (We prove the lemma for our particular setting though a more general result is true.)
\begin{lem}\label{lphardy}
Let $f=f_1+f_2$ where $f_1\in L^{\frac 43}(B_R)$ and $f_2 \in L\log L(B_R)$ be a compactly supported function with $\spt(f) \subset B_R$ and $\int_{\Real^4} f(x)dx =0$. Then $f\in \mathcal{H}^1(B_R)$ and there exists $C>0$ such that
\begin{equation}\label{HLest}
\|f\|_{\mathcal H^1(B_R)} \leq C\left(R\|f_1\|_{L^{\frac 43}(B_{R})}+ \|f_2\|_{L\log L(B_R)}+ |B_{2R}|\right).
\end{equation}
\end{lem}
\begin{proof}
First note that
\begin{align}
\|f_*\|_{L^1} = \int_{B_{2R}}f_*(x)dx +\int_{\Real^4 \backslash B_{2R}} f_* (x)dx. \label{lphardy1}
\end{align}
Since $f_1\in L^{\frac 43}(\Real^4)$ and $f_2 \in L\log L(\Real^4)$ we see that $f\in L^1_{{loc}}(\Real^4)$ and therefore $f_*(x)\le cMf(x)$ for every $x\in \Real^4$. Here $Mf:\Real^4 \to \Real$ is the maximal function defined by
\[
Mf(x)=\sup_{r>0} \frac{1}{|B_r(x)|} \int_{B_r(x)} |f(y)| dy.
\]Using the above, H\"older's inequality and the estimates $\|Mf_1\|_{L^{\frac 43}} \leq c \|f_1\|_{L^{\frac 43}}$, $\|Mf_2\|_{L^1(B_{2R})} \leq c \|f_2\|_{L \log L(B_{2R})}+c|B_{2R}|$,
\begin{align}
\int_{B_{2R}}f_*(x)dx%\le& \|(f_1)_*\|_{L^1} + \|(f_2)_*\|_{L^1} \nonumber \\
\le &cR\|(f_1)_*\|_{L^{\frac 43}} + \|(f_2)_*\|_{L^1} \nonumber \\
\le& c R\|Mf_1\|_{L^{\frac 43}}+ c\|Mf_2\|_{L^1}\nonumber \\
\le& cR\|f_1\|_{L^{\frac 43}}+c\|f_2\|_{L\log L}+c|B_{2R}|. \label{lphardy2}
\end{align}
Now we calculate for $\phi \in T$ and $x\in \Real^4$
\begin{align*} 
|\phi_t \star f(x)| &= |\int_{B_R} \phi_t(x-y)f(y)dy|\\
&= |\int_{B_R} (\phi_t(x-y)-\phi_t(x))f(y) dy| \\
&\le \|\nabla \phi_t\|_{L^\infty}\int_{B_R} |y| |f(y)|dy,
\end{align*}
where we used the mean value theorem and the cancellation property $\int_{\Real^4}f(y)dy=0$. Since $\|\nabla \phi_t\|_{L^\infty}\le \frac{1}{t^{5}}$, for $t>0$, we estimate
\begin{align}
|\phi_t \star f(x)| &\le \frac{R}{t^{5}}\int_{B_R} |f(y)|dy \nonumber \\
&\le \frac{cR^{2}}{t^{5}}\|f_1\|_{L^{\frac 43}} + \frac{cR}{t^{5}}\|f_2\|_{L\log L}. \label{lphardy3}
\end{align}
Assuming now that $|x|\ge 2R$ we can apply a technical result to get
\begin{align}
f_*(x)&= \sup_{\phi \in T} \sup_{t> \frac{|x|}{2}}|\phi_t \star f(x)| \nonumber \\
&\le \frac{cR^{2}}{|x|^{5}}\|f_1\|_{L^{\frac 43}}+\frac{cR}{|x|^{5}}\|f_2\|_{L\log L}. \label{lphardy4}
\end{align}
Inserting \eqref{lphardy2} and \eqref{lphardy4} into \eqref{lphardy1} we conclude
\begin{align}
\|f_*\|_{L^1} &\le cR\|f_1\|_{L^{\frac 43}}+ +c\|f_2\|_{L\log L}+c|B_{2R}|\\
& \quad \quad +\left(cR^{2}\|f_1\|_{L^{\frac 43}}+cR\|f_2\|_{L\log L}\right)\int_{\Real^4 \backslash B_{2R}} \frac{1}{|x|^{5}}dx\nonumber \\
&\le  cR\|f_1\|_{L^{\frac 43}}+ c\|f_2\|_{L \log L}+c|B_{2R}|. \label{lphardy5}
\end{align} 
\end{proof}

We also note two important inequalities (with proofs following those of \cite{ST}).
\begin{lem}\label{fL1lemma}
Let $f\in L\log L(B_r(x_0))$ for $r \in (0,1/2]$. There exists $C>0$ such that
\begin{equation}
\|f\|_{L^1(B_r(x_0))} \leq C \left(\log\left(\frac 1r\right)\right)^{-1} \|f\|_{L \log L(B_r(x_0))}.
\end{equation}
\end{lem}
\begin{proof}
Start by observing that
\begin{align*}
0 & \leq r^4 \int_0^{|B_1|} f^*(r^4 t) \log\left(2 + \frac 1t\right) dt\\
&= \int_0^{|B_r(x_0)|}f^*(s) \log \left(2 + \frac{r^4}s\right) ds\\
&= \int_0^{|B_r(x_0)|}f^*(s)\log(r^4) ds + \int_0^{|B_r(x_0)|}f^*(s) \log \left(\frac{2}{r^4} + \frac 1s\right) ds\\
& \leq -4 \log\left(\frac 1r\right)\|f\|_{L^1(B_r(x_0))} + C \|f\|_{L \log L(B_r(x_0))}.
\end{align*}The last inequality follows from the fact that there exists a fixed $C$ such that
\[
\frac {2}{r^4} + \frac 1s \leq \frac{2\omega_4 + 1}{s} \leq \left(2+ \frac 1s\right)^C
\]for all $s \leq \omega_4 r^4$. 
\end{proof}

Let $u$ be an $f$-approximate biharmonic map on $B_1$ with $f \in L\log L(B_1)$. For $x_0 \in B_1$ and $R >0$ such that $B_R(x_0) \subset B_1$ define $\hat u(x):= u(x_0 + Rx), \hat f(x) := R^4 f(x_0 + Rx)$. Then $\hat u$ is an $\hat f$-approximate biharmonic map. Moreover, we note that for any $r \in (0,1)$, $p\geq 1$, and $k=1,2,3$,
\begin{enumerate}
\item $\|D^k\hat u\|_{L^{4/k}(B_r)}=\|D^ku\|_{L^{4/k}(B_{rR}(x_0))}$.
\item $\|\hat f\|_{L^p(B_r)} = R^{4(1-1/p)}\|f\|_{L^p(B_{rR}(x_0))}$.
\end{enumerate}
\begin{lem}\label{fLlogLlemma}
Let $f \in L\log L(B_r(x_0))$ where $r \in(0, 1/2]$ and define $\hat f(x):= r^4f(x_0 + rx)$. Then there exists $C>0$ such that
\[
\|\hat f\|_{L \log L(B_1)} \leq C \|f\|_{L \log L(B_r(x_0))}.
\]
\end{lem}
\begin{proof}
First note that using the definition of $\hat f$, one can immediately show that $\hat f^*(t) = r^4 f^*(r^4t)$. Thus,
\begin{align*}
\int_0^{|B_1|} \hat f^*(t) \log \left(2 + \frac 1t\right) dt &= \int_0^{|B_1|} r^4 f^*(r^4t) \log\left(2 + \frac 1t\right) dt\\
&= \int_0^{|B_r(x_0)|} f^*(s) \log\left(2 + \frac{r^4}{s}\right) ds\\
& \leq \int_0^{|B_r(x_0)|} f^*(s) \log\left(2 + \frac{1}{s}\right) ds.
\end{align*}
\end{proof}
\bibliographystyle{amsplain}
\bibliography{Biblio}

\end{document}